\documentclass[12pt,a4paper]{wang}
\usepackage{etex}
\usepackage[numbers]{natbib}
\usepackage{amssymb,amstext,amscd,amsfonts,mathdots}
\numberwithin{equation}{section}
\usepackage{hyperref}
\usepackage[vcentermath]{youngtab}
\usepackage{setspace}
\usepackage[all]{xy}
\usepackage{tikz}
\usetikzlibrary{shapes.geometric, arrows}
\usetikzlibrary{calc}
\usepackage{amsthm}
\usepackage{enumerate}
\usepackage{enumitem}
\setlist[description]{leftmargin=\parindent,labelindent=\parindent}
\usepackage{diagbox}
\usepackage{multirow,multicol}
\usepackage{pdflscape}
\usepackage{indentfirst}
\usepackage{ytableau}
\usepackage{color}
\usepackage{colortbl}
\definecolor{mygray}{gray}{0.80}
\definecolor{myblue}{rgb}{0.50,0.50, 0.95}
\definecolor{myred}{rgb}{0.99, 0.51, 0.65}
\usepackage{amsmath}
\usepackage{lipsum}
\let\OLDthebibliography\thebibliography
\renewcommand\thebibliography[1]{\OLDthebibliography{#1}
\setlength{\parskip}{0pt}
\setlength{\itemsep}{0.5pt}}

\newtheorem{theorem}{Theorem}[section]
\newtheorem{theorem*}{Theorem}
\newtheorem{corollary}[theorem]{Corollary}
\newtheorem{corollary*}[theorem*]{Corollary}
\newtheorem{lemma}[theorem]{Lemma}
\newtheorem{proposition}[theorem]{Proposition}

\theoremstyle{definition}
\newtheorem{definition}[theorem]{Definition}
\newtheorem{remark}[theorem]{Remark}

\newtheorem*{question*}{Question}
\newtheorem{conjecture}[theorem]{Conjecture}
\newtheorem*{conjecture*}{Conjecture}
\newtheorem{example}[theorem]{Example}

\newtheorem*{notation*}{Notation}

\newtheorem*{claim*}{Claim}
\newtheorem{definitiontheorem}[theorem]{Definition-Theorem}

\begin{document}
\begin{spacing}{1.2}

\title{\textbf{On $\tau$-tilting finiteness of the Schur algebra}}
\author{Qi Wang\thanks{Q.W. is supported by the JSPS Grant-in-Aid for JSPS Fellows \text{20J10492}.}}

\abstract{\ \ \ \ We determined the $\tau$-tilting finiteness of Schur algebras over an algebraically closed field of arbitrary characteristic, except for a few small cases.}
\maketitle

\section{Introduction}
Throughout this paper, we always assume that $A$ is a finite-dimensional basic algebra over an algebraically closed field $\mathbb{F}$ of characteristic $p>0$. In particular, the representation type of $A$ is divided into representation-finite, (infinite-)tame and wild.

In recent years, $\tau$-tilting theory introduced by Adachi, Iyama and Reiten \cite{AIR} has drawn more and more attention. Here, $\tau$ is the Auslander-Reiten translation. The crucial concept of $\tau$-tilting theory is the notion of \emph{support $\tau$-tilting modules} (see Definition \ref{def-tau-tilting}), which can be regarded as a generalization of classical tilting modules from the viewpoint of mutation, that is, any basic almost complete support $\tau$-tilting module is a direct summand of exactly two basic support $\tau$-tilting modules. Moreover, support $\tau$-tilting modules are in bijection with many important objects in representation theory, such as two-term silting complexes, functorially finite torsion classes, left finite semibricks, etc. We refer to \cite{AIR} and \cite{Asai} for more details.

Similar to the representation finiteness, a modern notion named \emph{$\tau$-tilting finiteness} is introduced by Demonet, Iyama and Jasso \cite{DIJ-tau-tilting-finite}. We recall that an algebra $A$ is said to be $\tau$-tilting finite if it has only finitely many pairwise non-isomorphic basic support $\tau$-tilting modules. (There are several equivalent conditions in Definition \ref{def-tau-tilting-finite} and Proposition \ref{tau-tilting-finite-rigid}.) A typical example of $\tau$-tilting finite algebras is the class of representation-finite algebras. We remark that the $\tau$-tilting finiteness for several classes of algebras is determined, such as algebras with radical square zero \cite{Ada-rad-square-0}, preprojective algebras of Dynkin type \cite{Mizuno}, Brauer graph algebras \cite{AAC-brauer grapha}, biserial algebras \cite{Mousavand-biserial alg} and minimal wild two-point algebras \cite{W-two-point}. In addition, it has been proved in some cases that $\tau$-tilting finiteness coincides with representation-finiteness, including cycle-finite algebras \cite{MS-cycle-finite}, gentle algebras \cite{P-gentle}, cluster-tilted algebras \cite{Z-tilted}, simply connected algebras \cite{W-simply}, quasi-tilted algebras, locally hereditary algebras, etc. \cite{AHMW-joint}.

In this paper, we focus on (classical) Schur algebras, which play an important role in both Representation theory and Lie theory. This class of algebras arose in connection with the theory of polynomial representations over the general linear group $\mathsf{GL}_n(\mathbb{F})$ and then received widespread attention. So far, the representation theory of Schur algebras has developed very well and many derivatives appeared, such as $q$-Schur algebras, infinitesimal Schur algebras, Borel-Schur algebras and so on. In particular, the representation type of Schur algebras is completely determined by various authors, including Erdmann \cite{Erdmann-finite}, Xi \cite{Xi-schur}, Doty-Nakano \cite{DN-semisimple} and  Doty-Erdmann-Martin-Nakano \cite{DEMN-tame schur}. We summarize these results in Proposition \ref{summary}.

Let $n,r$ be two positive integers and $V$ an $n$-dimensional vector space over $\mathbb{F}$. We denote by $V^{\otimes r}$ the $r$-fold tensor product $V\otimes_\mathbb{F}V\otimes_\mathbb{F}\dots\otimes_\mathbb{F}V$. Then, the symmetric group $G_r$ has a natural action (by permutation) on $V^{\otimes r}$ which makes it to be a module over the group algebra $\mathbb{F}G_r$ of $G_r$. Then, the endomorphism ring $\mathsf{End}_{\mathbb{F}G_r}\left ( V^{\otimes r} \right )$ is called the \emph{Schur algebra} (see \cite[Section 2]{Martin-schur alg}) and we denote it by $S(n,r)$.

Our initial motivation is to check whether $\tau$-tilting finiteness and representation-finiteness coincide in the class of Schur algebras. It is clear that representation-finite Schur algebras are $\tau$-tilting finite. Moreover, we have already determined the number of pairwise non-isomorphic basic support $\tau$-tilting modules blockwise in Theorem \ref{result-finite}. Therefore, we start with the tame cases.
\begin{theorem}[Theorem \ref{result-tame}]
If the Schur algebra $S(n,r)$ is tame, it is $\tau$-tilting finite.
\end{theorem}

In order to prove the above result, we may show that any block algebra of tame Schur algebras is $\tau$-tilting finite. Let $S(n,r)$ be a tame Schur algebra and $A:=eS(n,r)e$ an indecomposable block algebra with a central idempotent $e$ of $S(n,r)$. It is known from \cite{Erdmann-finite} and \cite[Section 5]{DEMN-tame schur} that the block algebra $A$ is Morita equivalent to one of $\mathbb{F}$, $\mathcal{A}_2$, $\mathcal{D}_3$, $\mathcal{D}_4$, $\mathcal{R}_4$ and $\mathcal{H}_4$ (see Section \ref{section-finite-tame} for the definitions). Then, we have

\begin{theorem}[Theorem \ref{result-finite} and Lemma \ref{method of D_4}]
Let $\mathsf{s\tau\text{-}tilt}\ A$ be the set of pairwise non-isomorphic basic support $\tau$-tilting $A$-modules. Then,
\begin{center}
\renewcommand\arraystretch{1.2}
\begin{tabular}{|c|c|c|c|c|c|}
\hline
$A$&$\mathcal{A}_2$&$\mathcal{D}_3$&$\mathcal{D}_4$&$\mathcal{R}_4$&$\mathcal{H}_4$\\ \hline
$\#\mathsf{s\tau\text{-}tilt}\ A$&6&28&114&88&96 \\ \hline
\end{tabular}.
\end{center}
\end{theorem}

Since the above results give a negative answer to our initial motivation, we next ask whether all wild Schur algebras are $\tau$-tilting infinite. There are a few cases in $(\star)$ that we currently do not have a suitable method to deal with. Since the number of pairwise non-isomorphic basic support $\tau$-tilting modules for them is huge, we cannot confirm the $\tau$-tilting finiteness of these exceptions by direct calculation. So, we leave these cases here and expect to find a new reduction method (for example, Conjecture \ref{conjecture-cartan}) to show the $\tau$-tilting finiteness.
\begin{center}
$(\star)\ \left\{\begin{aligned}
p&=2,n=2, r=8, 17, 19; \\
p&=2,n=3, r=4;\\
p&=2,n\geqslant 5,r=5; \\
p&\geqslant 5, n=2,p^2\leqslant r\leqslant p^2+p-1.
\end{aligned}\right.$
\end{center}

Now, we are able to determine the $\tau$-tilting finiteness for most wild Schur algebras.
\begin{theorem}[Table \ref{p=2}, Table \ref{p=3} and Table \ref{p=5}]
Let $p$ be the characteristic of $\mathbb{F}$. Except for the cases in $(\star)$, the wild Schur algebra $S(n, r)$ is $\tau$-tilting finite if and only if $p=2$, $n=2, r=6,13,15$ or $n=3, r=5$ or $n=4,r=4$.
\end{theorem}

We have the following answers for our motivations.
\begin{corollary}
Let $p=3$. Then, $S(n,r)$ is $\tau$-tilting infinite if and only if it is wild.
\end{corollary}
\begin{corollary}
Let $p\geqslant 5$. Except for the cases in $(\star)$, $S(n,r)$ is $\tau$-tilting finite if and only if it is representation-finite.
\end{corollary}

Actually, except for the cases in $(\star)$, we have observed from the proof process that the Schur algebra $S(n,r)$ is $\tau$-tilting infinite if and only if it has one of $\tau$-tilting infinite quivers $\mathsf{Q}_1, \mathsf{Q}_2$ and $\mathsf{Q}_3$ (see Lemma \ref{tau-tilting infinite quiver}) as a subquiver. This leads us to conjecture that all cases in $(\star)$ are $\tau$-tilting finite, because all the cases in $(\star)$ do not have one of $\mathsf{Q}_1, \mathsf{Q}_2$ and $\mathsf{Q}_3$ as a subquiver.

This paper is organized as follows. In Section \ref{section2}, we first recall some basic materials on $\tau$-tilting theory and the Schur algebra $S(n,r)$. Then, we give several reduction theorems on the $\tau$-tilting finiteness of Schur algebras, such that we only need to consider small $n$ and $r$. At the end of the section, we explain our strategy to prove $S(n,r)$ to be $\tau$-tilting infinite. In Section \ref{section-finite-tame}, we show that all tame Schur algebras are $\tau$-tilting finite. In Section \ref{section-wild}, we determine the $\tau$-tilting finiteness of most wild Schur algebras.

\ \\
\noindent\textit{\textbf{Acknowledgments.}} This paper is inspired by another on-going joint work with Susumu Ariki,  Ryoichi Kase, Kengo Miyamoto and Euiyong Park, where the authors are using $\tau$-tilting theory to classify $\tau$-tilting finite block algebras of Hecke algebras of classical type. I am very grateful to them, especially, Prof. Ariki, for giving me a lot of kind guidance in writing this paper. I am also very grateful to the anonymous referee for giving me many useful suggestions.

\section{Preliminaries}\label{section2}
We recall that any finite-dimensional basic algebra $A$ over an algebraically closed field $\mathbb{F}$ is isomorphic to a bound quiver algebra $\mathbb{F}Q/I$, where $\mathbb{F}Q$ is the path algebra of the Gabriel quiver $Q=Q_A$ of $A$ and $I$ is an admissible ideal of $\mathbb{F}Q$. We refer to \cite{ASS} for more background materials on quiver and representation theory of algebras.

\subsection{$\tau$-tilting theory}
In this subsection, we review some background of $\tau$-tilting theory. However, we only review the background needed for this paper, so that many aspects related to $\tau$-tilting theory are omitted. One may look at \cite{AIR}, \cite{Asai} and \cite{AIRRT} for more materials.

Let $\mathsf{mod}\ A$ be the category of finitely generated right $A$-modules and $\mathsf{proj}\ A$ the category of finitely generated projective right $A$-modules. For any $M\in \mathsf{mod}\ A$, we denote by $\mathsf{add}(M)$ (respectively, $\mathsf{Fac}(M)$) the full subcategory of $\mathsf{mod}\ A$ whose objects are direct summands (respectively, factor modules) of finite direct sums of copies of $M$. We denote by $A^{\mathsf{op}}$ the opposite algebra of $A$ and by $|M|$ the number of isomorphism classes of indecomposable direct summands of $M$. In particular, we often describe $A$-modules via their composition factors. For example, each simple $A$-module $S_i$ is written as $i$, and then $\substack{S_1\\S_2}=\substack{1\\2}$ is an indecomposable $A$-module $M$ with a unique simple submodule $S_2$ such that $M/S_2\simeq S_1$.

We recall that the Nakayama functor $\nu:=D(-)^\ast$ is induced by the dualities
\begin{center}
$D:=\mathsf{Hom}_\mathbb{F}(-,\mathbb{F}): \mathsf{mod}\ A\leftrightarrow  \mathsf{mod}\ A^{\mathsf{op}}$ and $(-)^\ast:=\mathsf{Hom}_A(-,A): \mathsf{proj}\ A\leftrightarrow  \mathsf{proj}\ A^{\mathsf{op}}$.
\end{center}
Then, for any $M\in \mathsf{mod}\ A$ with a minimal projective presentation
\begin{center}
$P_1\overset{f_1}{\longrightarrow }P_0\overset{f_0}{\longrightarrow }M\longrightarrow 0$,
\end{center}
the Auslander-Reiten translation $\tau M$ is defined by the exact sequence:
\begin{center}
$0\longrightarrow \tau M\longrightarrow \nu P_1\overset{\nu f_1}{\longrightarrow} \nu P_0$.
\end{center}

\begin{definition}(\cite[Definition 0.1]{AIR})\label{def-tau-tilting}
Let $M\in \mathsf{mod}\ A$.
\begin{description}\setlength{\itemsep}{-3pt}
  \item[(1)] $M$ is called $\tau$-rigid if $\mathsf{Hom}_A(M,\tau M)=0$.
  \item[(2)] $M$ is called $\tau$-tilting if $M$ is $\tau$-rigid and $\left | M \right |=\left | A \right |$.
  \item[(3)] $M$ is called support $\tau$-tilting if there exists an idempotent $e$ of $A$ such that $M$ is a $\tau$-tilting $\left ( A/A e A\right )$-module.
\end{description}
\end{definition}

In the case of support $\tau$-tilting modules, we may correspondingly define support $\tau$-tilting pairs. For a pair $(M,P)$ with $M\in \mathsf{mod}\ A$ and $P\in \mathsf{proj}\ A$, it is called a support $\tau$-tilting pair if $M$ is $\tau$-rigid, $\mathsf{Hom}_A(P,M)=0$ and $|M|+|P|=|A|$. Obviously, $(M,P)$ is a support $\tau$-tilting pair if and only if $M$ is a $\tau$-tilting $\left ( A/Ae A \right )$-module and $P=eA$.

\begin{definition}\label{def-tau-tilting-finite}
An algebra $A$ is called $\tau$-tilting finite if it has only finitely many pairwise non-isomorphic basic $\tau$-tilting modules. Otherwise, $A$ is called $\tau$-tilting infinite.
\end{definition}

We denote by $\mathsf{\tau\text{-}rigid}\ A$ (respectively, $\mathsf{s\tau\text{-}tilt}\ A$) the set of isomorphism classes of indecomposable $\tau$-rigid (respectively, basic support $\tau$-tilting) $A$-modules. It is known from \cite[Theorem 0.2]{AIR} that any $\tau$-rigid $A$-module is a direct summand of some $\tau$-tilting $A$-module. Moreover, we have the following statement.
\begin{proposition}{\rm{(\cite[Corollary 2.9]{DIJ-tau-tilting-finite})}}\label{tau-tilting-finite-rigid}
An algebra $A$ is $\tau$-tilting finite if and only if one of (equivalently, any of) the sets $\mathsf{\tau\text{-}rigid}\ A$ and  $\mathsf{s\tau\text{-}tilt}\ A$ is a finite set.
\end{proposition}

We recall the concept of left mutation, which is the core concept of $\tau$-tilting theory. Before doing this, we need the definition of minimal left approximation. Let $\mathcal{C}$ be an additive category and $X, Y$ objects of $\mathcal{C}$. A morphism $f: X\rightarrow Z$ with $Z\in \mathsf{add}(Y)$ is called a minimal left $\mathsf{add}(Y)$-approximation of $X$ if it satisfies:
\begin{itemize}\setlength{\itemsep}{-3pt}
  \item every $h\in \mathsf{Hom}_\mathcal{C}(Z,Z)$ that satisfies $h\circ f=f$ is an automorphism,
  \item $\mathsf{Hom}_\mathcal{C}(f,Z'): \mathsf{Hom}_\mathcal{C}(Z,Z')\rightarrow \mathsf{Hom}_\mathcal{C}(X,Z')$ is surjective for any $Z'\in \mathsf{add}(Y)$,
\end{itemize}
where $\mathsf{add}(Y)$ is the category of all direct summands of finite direct sums of copies of $Y$.

\begin{definitiontheorem}(\cite[Definition 2.19, Theorem 2.30]{AIR})\label{def-left-mutation}
Let $T=M\oplus N$ be a basic support $\tau$-tilting $A$-module with an indecomposable direct summand $M$ satisfying $M\notin \mathsf{Fac}(N)$. We take an exact sequence with a minimal left $\mathsf{add}(N)$-approximation $f$:
\begin{center}
$M\overset{f}{\longrightarrow}N'\longrightarrow \mathsf{coker}\ f \longrightarrow 0$.
\end{center}
We call $\mu_M^-(T):=(\mathsf{coker}\ f )\oplus N$ the left mutation of $T$ with respect to $M$, which is again a basic support $\tau$-tilting $A$-module. (The right mutation $\mu_M^+(T)$ can be defined dually.)
\end{definitiontheorem}

We may construct a directed graph $\mathcal{H}(\mathsf{s\tau\text{-}tilt}\ A)$ by drawing an arrow from $T_1$ to $T_2$ if $T_2$ is a left mutation of $T_1$. On the other hand, we can regard $\mathsf{s\tau\text{-}tilt}\ A$ as a poset with respect to a partial order $\leq $. For any $M, N\in \mathsf{s\tau\text{-}tilt}\ A$, let $(M,P)$ and $(N,Q)$ be their corresponding support $\tau$-tilting pairs, respectively. We say that $N\leq M$ if $\mathsf{Fac}(N) \subseteq \mathsf{Fac}(M)$ or, equivalently, $\mathsf{Hom}_A(N,\tau M)=0$ and $\mathsf{add}(P) \subseteq \mathsf{add}(Q)$. Then, $\mathcal{H}(\mathsf{s\tau\text{-}tilt}\ A)$ is exactly the Hasse quiver on the poset $\mathsf{s\tau\text{-}tilt}\ A$, see \cite[Corollary 2.34]{AIR}.

The following statement implies that an algebra $A$ is $\tau$-tilting finite if we can find a finite connected component of $\mathcal{H}(\mathsf{s\tau\text{-}tilt}\ A)$.
\begin{proposition}{\rm{(\cite[Corollary 2.38]{AIR})}}
If the Hasse quiver $\mathcal{H}(\mathsf{s\tau\text{-}tilt}\ A)$ contains a finite connected component, it exhausts all support $\tau$-tilting $A$-modules.
\end{proposition}

According to the duality $(-)^\ast=\mathsf{Hom}_A(-,A)$, we have
\begin{proposition}{\rm{(\cite[Theorem 2.14]{AIR})}}\label{oppsite algebra}
There exists a poset anti-isomorphism between $\mathsf{s\tau\text{-}tilt}\ A$ and $\mathsf{s\tau\text{-}tilt}\ A^{\mathsf{op}}$.
\end{proposition}

We may give an example to illustrate the constructions above.
\begin{example}\label{example-simplest}
Let $A:=\mathbb{F}(\xymatrix{1\ar@<0.5ex>[r]^{\alpha} & 2 \ar@<0.5ex>[l]^{\beta}})/<\alpha\beta,\beta\alpha>$. We denote by $S_1, S_2$ the simple $A$-modules and by $P_1, P_2$ the indecomposable projective $A$-modules. We take an exact sequence with a minimal left $\mathsf{add}(P_1)$-approximation $f$ of $P_2$:
\begin{center}
$\substack{ e_2 \\ \beta }\overset{\alpha \cdot }{\longrightarrow}\substack{e_1\\ \alpha }\longrightarrow \mathsf{coker}\ f\longrightarrow 0$.
\end{center}
Then, $\mathsf{coker}\ f=S_1$ and $\mu_{P_2}^-(A)= P_1\oplus S_1$. Similarly, $\mu_{P_1}^-(A)= S_2\oplus P_2$ and we can compute the left mutations step by step, so that the Hasse quiver $\mathcal{H}(\mathsf{s\tau\text{-}tilt}\ A)$ is
\begin{center}
$\vcenter{\xymatrix@C=1.2cm@R=0.1cm{& P_1\oplus S_1\ar[r]^{>} &S_1\ar[dr]^{>}\\P_1\oplus P_2\ar[dr]^{>}\ar[ur]^{>}&&&0 \\&S_2\oplus P_2\ar[r]^{>}&S_2\ar[ur]^{>}&}}$,
\end{center}
where $>$ is the partial order on $\mathsf{s\tau\text{-}tilt}\ A$. Then, we have
\begin{itemize}\setlength{\itemsep}{-3pt}
\item $\mathsf{\tau\text{-}rigid}\ A=\{P_1, P_2, S_1, S_2, 0\}$ and $\mathsf{s\tau\text{-}tilt}\ A=\{P_1\oplus P_2, P_1\oplus S_1, S_2\oplus P_2, S_1,S_2,0\}$. In particular, $P_1\oplus P_2, P_1\oplus S_1$ and $S_2\oplus P_2$ are $\tau$-tilting $A$-modules.
\item $(S_1,P_2)$ and $(S_2,P_1)$ are support $\tau$-tilting pairs.
\end{itemize}
\end{example}

\subsection{The symmetric group and the Schur algebra}\label{sub2.1}
We refer to some textbooks, such as \cite{James-symmetric}, \cite{Martin-schur alg} and \cite{Sagan-symmetric}, for the representation theory of the symmetric group and the Schur algebra.

Let $r$ be a natural number and $\lambda:=(\lambda_1,\lambda_2,\ldots)$ a sequence of non-negative integers. We call $\lambda$ a partition of $r$ if $\sum_{i\in\mathbb{N}}^{}\lambda_i=r$ with $\lambda_1\geqslant \lambda_2\geqslant \ldots \geqslant 0$, and the elements $\lambda_i$ are called parts of $\lambda$. If there exists an $n\in \mathbb{N}$ such that $\lambda_i=0$ for all $i>n$, we denote $\lambda$ by $(\lambda_1,\lambda_2,\ldots,\lambda_n)$ and call it a partition of $r$ with at most $n$ parts. We denote by $\Omega(r)$ the set of all partitions of $r$ and by $\Omega(n,r)$ the set of all partitions of $r$ with at most $n$ parts. It is well-known that $\Omega(r)$ admits the dominance order $\triangleright$ and the lexicographic order $>$, we omit the definitions.

We denote by $G_r$ the symmetric group on $r$ symbols and by $\mathbb{F}G_r$ the group algebra of $G_r$. Each partition $\lambda$ of $r$ gives a Young subgroup $G_\lambda$ of $G_r$ defined as
\begin{center}
$G_\lambda:=G_{\{1,2,\dots, \lambda_1\}}\times G_{\{\lambda_1+1,\lambda_1+2,\dots, \lambda_1+\lambda_2\}}\times \dots \times G_{\{\lambda_1+\dots+\lambda_{n-1}+1,\lambda_1+\dots+\lambda_{n-1}+2,\dots, r\}}$.
\end{center}
Then, the \emph{permutation $\mathbb{F}G_r$-module} $M^\lambda$ is $1_{G_\lambda}\uparrow^{G_r}$, where $1_{G_\lambda}$ denotes the trivial module for $G_\lambda$ and $\uparrow$ denotes induction.

Let $S(n,r)=\mathsf{End}_{\mathbb{F}G_r}\left ( V^{\otimes r} \right )$ be the Schur algebra. It is well-known (e.g., see \cite[Section 1.6]{Ve-doc-thesis}) that $M^\lambda$ with $\lambda\in \Omega(n,r)$ can be regarded as direct summands of $V^{\otimes r}$. Therefore, we have the following algebra isomorphism,
\begin{center}
$S(n,r)\simeq \mathsf{End}_{\mathbb{F}G_r}\left ( \underset{\lambda\in \Omega(n,r)}{\bigoplus} n_\lambda M^\lambda \right )$,
\end{center}
where $1\leqslant n_\lambda \in \mathbb{N}$ is the number of compositions of $r$ with at most $n$ parts which are rearrangement of $\lambda$.

\subsubsection{Specht modules}
We follow the conventions in \cite{James-symmetric} for the constructions of Specht modules.

We may regard a partition $\lambda$ of $r$ as a box-diagram $[\lambda]$ of which the $i$-th row contains $\lambda_i$-boxes. For a prime number $p$, a partition $\lambda$ (or a diagram $[\lambda]$) is called $p$-regular if no $p$ rows of $\lambda$ have the same length. Otherwise, $\lambda$ (or $[\lambda]$) is called $p$-singular.

Let $\lambda$ be a partition of $r$. A $\lambda$-tableau $t$ is obtained from $[\lambda]$ by filling the boxes by numbers $\{1,2,\dots,r\}$ without repetition. In fact, $t$ is a bijection between the boxes in $[\lambda]$ and the numbers in $\{1,2,\dots,r\}$. For any $\sigma\in G_r$, we define an action $t\cdot \sigma:=t\circ \sigma$ by the composition of the bijection $t$ and the permutation $\sigma$. Then, the column stabilizer of a $\lambda$-tableau $t$ is defined as the subgroup $C_t$ of $G_r$ consisting of permutations preserving the numbers in each column of $t$. Similarly, the row stabilizer of $t$ is the subgroup $R_t$ consisting of permutations preserving the numbers in each row of $t$.

Let $t, t'$ be two $\lambda$-tableaux. We may define a row-equivalence relation $t\sim t'$ if $t'=t\cdot \sigma$ for $\sigma\in R_t$. We denote by $\{t\}$ the equivalence class of $t$ under $\sim$ and call it a $\lambda$-tabloid. We also define a $G_r$-action on a $\lambda$-tabloid $\{t\}$ by $\{t\}\cdot \sigma:=\{t\cdot \sigma\}$ for any $\sigma\in G_r$, and this action is well-defined. Then, the $\lambda$-polytabloid $e_t$ associated with a $\lambda$-tableau $t$ is defined by $e_t:=\{t\}\cdot \kappa_t$, where $\kappa_t:=\sum\limits_{\sigma\in C_t}^{}\text{sgn}(\sigma)\sigma$ is the signed column sum.

Let $M^\lambda$ be the permutation module of $\mathbb{F}G_r$ corresponding to a partition $\lambda$. It is known that $M^\lambda$ is spanned by all $\lambda$-tabloids. Then, we call the submodule $S^\lambda$ of $M^\lambda$ spanned by all $\lambda$-polytabloids the \emph{Specht module} corresponding to $\lambda$. If $\mathbb{F}$ is a field of characteristic zero, the set $\left \{ S^\lambda \mid \lambda\in \Omega(r) \right \}$ is a complete set of pairwise non-isomorphic simple $\mathbb{F}G_r$-modules. If $\mathbb{F}$ is a field of prime characteristic $p$, each Specht module $S^\lambda$ with $\lambda$ being $p$-regular has a unique (up to isomorphism) simple top and we may denote it by $D^\lambda$. Then, the set $\left \{ D^\lambda \mid \lambda\in \Omega(r), \lambda\ \text{is}\ p\text{-regular} \right \}$ is a complete set of pairwise non-isomorphic simple $\mathbb{F}G_r$-modules. (In the case of a $p$-singular partition $\mu$, all of the composition factors of $S^\mu$ are $D^\lambda$ such that $\lambda$ is a $p$-regular partition which dominates $\mu$.)

Let $\mathbb{F}$ be a field of prime characteristic $p$. The decomposition number $[S^\lambda:D^\mu]$ provides how many times each simple module $D^\mu$ occurs as a composition factor of the Specht module $S^\lambda$. If we run all partitions of $r$, we get the decomposition matrix of $\mathbb{F}G_r$. We recall from \cite[Corollary 12.3]{James-symmetric} that the decomposition matrix has the following form.
\begin{center}
$\begin{matrix}
\\ \\
S^{\lambda}, \lambda\ p\text{-regular}\left\{\begin{matrix}
\\
\\
\\
\\
\\
\end{matrix}\right.\\
\\ S^{\lambda}, \lambda\ p\text{-singular}\left\{\begin{matrix}
\\
\\
\end{matrix}\right.
\end{matrix}\begin{matrix}D^{\mu}, \mu\ p\text{-regular}\\\overbrace{\begin{pmatrix}
1&  &  & &\\
*&1& & O& \\
*&*&1\\
\vdots&\vdots&\vdots&\ddots& \\
*&*  &*&\dots  & 1\\
-&-&-&-&-\\
*&*  &*&\dots  & *\\
*&*  &*&\dots  & *
\end{pmatrix}}\end{matrix}$
\end{center}
Although it is still an unsolved problem to determine the decomposition matrix of $\mathbb{F}G_r$, James \cite{James-symmetric} provides us with enough materials to write this paper.

\subsubsection{Young modules}
We recall that a permutation module $M^\lambda$ over $\mathbb{F}$ is liftable by a $p$-modular system and therefore, it has an associated ordinary character $\mathsf{ch}\ M^\lambda$. Let $\chi^\lambda$ be the ordinary character corresponding to Specht module $S^\lambda$ over a field of characteristic zero. Then, $\chi^\lambda$ is a constituent of $\mathsf{ch}\ M^\lambda$ and $\chi^\mu (\mu\neq \lambda)$ is a constituent of $\mathsf{ch}\ M^\lambda$ if and only if $\mu\triangleright \lambda$.

We decompose $M^\lambda$ into some indecomposable direct summands $\oplus_{i=1}^nY_i$ for $n\in \mathbb{N}$. Obviously, each summand $Y_i$ is also liftable and has an associated ordinary character $\mathsf{ch}\ Y_i$. Then, the unique (up to isomorphism) direct summand $Y_i$ which the ordinary character $\chi^\lambda$ occurs in $\mathsf{ch}\ Y_i$, is called the \emph{Young module} corresponding to $\lambda$ and we denote it by $Y^\lambda$. It is well-known that the Young module $Y^\lambda$ has a Specht filtration given by $Y^\lambda=Z_1\supseteq Z_2\supseteq \dots\supseteq Z_k=0$ for some $k\in \mathbb{N}$, where each $Z_i/Z_{i+1}$ is isomorphic to a Specht module $S^\mu$ with $\mu\trianglerighteq \lambda$. Moreover, the Young module $Y^\lambda$ is self-dual, that is, $Y^\lambda\simeq D(Y^\lambda)$ with respect to $D=\mathsf{Hom}_\mathbb{F}(-,\mathbb{F})$. In fact, $D(Y^\lambda)$ becomes a right $\mathbb{F}G_r$-module via $(f\cdot \sigma)(x):=f(x\sigma^{-1})$ for $f\in D(Y^\lambda)$, $\sigma\in G_r$ and $x\in Y^\lambda$.

\begin{proposition}{\rm{(\cite[Section 4.6]{Martin-schur alg})}}
The set $\left \{ Y^\lambda \mid \lambda\in \Omega(n,r) \right \}$ is a complete set of pairwise non-isomorphic Young modules which occurs as indecomposable direct summands of permutation modules in $\left \{ M^\lambda \mid \lambda\in \Omega(n,r) \right \}$.
\end{proposition}

It is worth mentioning that if $\lambda$ is a partition with at most two parts, Henke \cite{H-character} provided a formula to calculate $\mathsf{ch}\ Y^\lambda$. We recall these constructions as follows.

Let $p$ be a prime number. There is a $p$-adic decomposition $s=\sum_{s\geqslant 0}^{}s_kp^k$ for any non-negative integer $s$. Now, let $s,t$ be two non-negative integers, we define a function
\begin{center}
$f(s,t)=\prod\limits_{k\in\{0\}\cup \mathbb{N}}^{}\begin{pmatrix}
p-1-s_k\\
p-1-t_k
\end{pmatrix}$,
\end{center}
where we set $\binom{m}{n}=0$ if $m<n$. Moreover, we define
\begin{center}
$g(s,t):=\left\{\begin{aligned}
1 &\ \text{ if } f(2t,s+t)=1,\\
0 &\ \text{ otherwise.}
\end{aligned}\right.$ and $h(s,t):=\left\{\begin{aligned}
1 &\ \text{ if } f(2t+1,s+t+1)=1,\\
0 &\ \text{ otherwise.}
\end{aligned}\right.$
\end{center}

\begin{theorem}{\rm{(\cite[Section 5.2]{H-character})}}\label{henke charac}
Let $(r-k,k)$ be a partition with a non-negative integer $k$ and $\mathsf{ch}\ Y^{(r-k,k)}$ the associated ordinary character of $Y^{(r-k,k)}$.
\begin{description}\setlength{\itemsep}{-3pt}
  \item[(1)] If $r$ is even, then
\begin{center}
  $\mathsf{ch}\ Y^{(r-k,k)}=\sum\limits_{i=0}^{\frac{r}{2}}g(\frac{r}{2}-i,\frac{r}{2}-k)\chi^{(r-i,i)}$.
\end{center}
  \item[(2)] If $r$ is odd, then
\begin{center}
  $\mathsf{ch}\ Y^{(r-k,k)}=\sum\limits_{i=0}^{[\frac{r}{2}]}h([\frac{r}{2}]-i,[\frac{r}{2}]-k)\chi^{(r-i,i)}$,
\end{center}
where $[\frac{r}{2}]$ is the greatest integer less than or equal to $\frac{r}{2}$.
\end{description}
\end{theorem}

We now explain how to construct the basic algebra of the Schur algebra $S(n,r)$. Let $B$ be a block of the group algebra $\mathbb{F}G_r$ labeled by a $p$-core $\omega$. It is well-known that a partition $\lambda$ belongs to $B$ if and only if $\lambda$ has the same $p$-core $\omega $. Then, we define
\begin{center}
$S_B:=\mathsf{End}_{\mathbb{F}G_r}\left ( \underset{\lambda\in B\cap\Omega(n,r)}{\bigoplus} Y^\lambda \right )$
\end{center}
and the basic algebra of $S(n,r)$ is $\bigoplus S_B$, where the sum is taken over all blocks of $\mathbb{F}G_r$. Moreover, $S_B$ is a direct sum of blocks of the basic algebra of $S(n,r)$. We remark that if we consider the Young modules $Y^\lambda$ for partitions $\lambda$ of $r$ with at most $n$ parts, the Specht modules $S^\mu$ in the Specht filtration of $Y^\lambda$ and the composition factors $D^\mu$ which appear in $Y^\lambda$ are also corresponding to the partitions with at most $n$ parts. (The reason is that $[S^\lambda:D^\mu]\neq 0 \Rightarrow \mu \unrhd \lambda$.)

We may give an example to illustrate our constructions above.
\begin{example}\label{s(2,11)}
We look at the Schur algebra $S(2,11)$ over $p=2$. Let $B_1$ be the principal block of $\mathbb{F}G_{11}$ and $B_2$ the block of $\mathbb{F}G_{11}$ labeled by $2$-core $(2,1)$. Then, we may find in \cite{James-symmetric} that the parts of the decomposition matrix $[S^\lambda:D^\mu]$ for the partitions in $B_1$ and $B_2$ with at most two parts are
\begin{center}
$B_1: \begin{matrix}
(11)\\
(9,2)\\
(7,4)
\end{matrix}\begin{pmatrix}
1&\\
0&1&\\
1&0&1
\end{pmatrix}$, $B_2: \begin{matrix}
(10,1)\\
(8,3)\\
(6,5)
\end{matrix}\begin{pmatrix}
1&\\
1&1&\\
0&1&1
\end{pmatrix}$.
\end{center}

We determine $S_{B_2}$ as follows. By using the formula given in Theorem \ref{henke charac}, we have
\begin{center}
$\begin{aligned}
\mathsf{ch}\ Y^{(10,1)}&=\chi^{(10,1)},\\
\mathsf{ch}\ Y^{(8,3)}&=\chi^{(10,1)}+\chi^{(8,3)},\\
\mathsf{ch}\ Y^{(6,5)}&=\chi^{(10,1)}+\chi^{(8,3)}+\chi^{(6,5)}.
\end{aligned}$
\end{center}
Similar to the proof of \cite[Lemma 4.4]{Erdmann-finite}, we may read off the Specht filtration of $Y^\lambda$ from the formula, and the composition factors of Young modules are $\{D^{(10,1)}, D^{(8,3)}, D^{(6,5)}\}$. It is obvious that $Y^{(10,1)}=S^{(10,1)}=D^{(10,1)}$. By the decomposition matrix above, the Specht module $S^{(8,3)}$ has composition factors $\{D^{(10,1)}, D^{(8,3)}\}$. Since the top of $S^{(8,3)}$ is $D^{(8,3)}$ and $S^{(8,3)}$ is a submodule of $Y^{(8,3)}$, the simple module $D^{(10,1)}$ is in the socle of $Y^{(8,3)}$. We deduce the radical series of $Y^{(8,3)}$ by using the self-duality of Young modules, that is,
\begin{center}
$Y^{(8,3)}=\begin{matrix}
S^{(10,1)}\\
S^{(8,3)}
\end{matrix}=\begin{matrix}
D^{(10,1)}\\
D^{(8,3)}\\
D^{(10,1)}
\end{matrix}$.
\end{center}
Similarly, the simple module $D^{(10,1)}$ appears in the top of $Y^{(6,5)}$ because $Y^{(6,5)}$ has Specht filtration whose top is $S^{(10,1)}$. As $(6,5)$ is $2$-regular, the top of $S^{(6,5)}$ is $D^{(6,5)}$ and the socle of $S^{(6,5)}$ is $D^{(8,3)}$. Since $S^{(6,5)}$ is the bottom Specht module, it implies that $D^{(8,3)}$ appears in the socle of $Y^{(6,5)}$. By the self-duality of Young modules, we deduce that
\begin{center}
$Y^{(6,5)}=\vcenter{\xymatrix@C=0.01cm@R=0.2cm{&D^{(8,3)}\ar@{-}[dr]\ar@{-}[ddl]&&&D^{(10,1)}\ar@{-}[ddl] \\
& &D^{(6,5)}\ar@{-}[dr]&&\\
D^{(10,1)}&&&D^{(8,3)}&}}$.
\end{center}
Thus, $S_{B_2}=\mathsf{End}_{\mathbb{F}G_{11}}(Y^{(10,1)}\oplus Y^{(8,3)}\oplus Y^{(6,5)})$ is isomorphic to $\mathbb{F}Q/I$ with
\begin{center}
$Q:\xymatrix@C=1cm{(10,1)\ar@<0.5ex>[r]^{\alpha_1}&(6,5)\ar@<0.5ex>[l]^{\beta_1}\ar@<0.5ex>[r]^{\alpha_2}&(8,3)\ar@<0.5ex>[l]^{\beta_2}}$ and $I: \left \langle \begin{matrix}
\alpha_1\beta_1,\beta_2\alpha_2, \alpha_1\alpha_2\beta_2,\alpha_2\beta_2\beta_1
\end{matrix}\right \rangle$,
\end{center}
where we replace each vertex in $Q$ by $\lambda$ associated with the Young module $Y^\lambda$.

Similarly, one may find that $S_{B_1}$ is isomorphic to $\mathbb{F}Q/I\oplus \mathbb{F}$, where
\begin{center}
$Q:\xymatrix@C=1cm{(11)\ar@<0.5ex>[r]^{\alpha_1}&(7,4)\ar@<0.5ex>[l]^{\beta_1}}$ and $I: \left \langle \begin{matrix}
\alpha_1\beta_1
\end{matrix}\right \rangle$.
\end{center}
Therefore, the basic algebra of $S(2,11)$ over $p=2$ is $S_{B_1}\oplus S_{B_2}$.
\end{example}

\subsection{Reduction theorems}
We first recall some reduction theorems for the $\tau$-tilting finiteness of an algebra $A$.
\begin{proposition}{\rm{(\cite[Theorem 5.12]{AIRRT}, \cite[Theorem 4.2]{DIJ-tau-tilting-finite})}}\label{quotient and idempotent}
If $A$ is $\tau$-tilting finite,
\begin{description}\setlength{\itemsep}{-3pt}
  \item[(1)] the quotient algebra $A/I$ is $\tau$-tilting finite for any two-sided ideal $I$ of $A$,
  \item[(2)] the idempotent truncation $eAe$ is $\tau$-tilting finite for any idempotent $e$ of $A$.
\end{description}
\end{proposition}

In addition, Eisele, Janssens and Raedschelders \cite{EJR18} provided us with a powerful reduction theorem, as shown below.
\begin{proposition}{\rm{(\cite[Theorem 1]{EJR18})}}\label{center}
Let $I$ be a two-sided ideal generated by central elements which are contained in the Jacobson radical $\mathsf{rad}\ A$ of $A$. Then, there exists a poset isomorphism between $\mathsf{s\tau\text{-}tilt}\ A$ and $\mathsf{s\tau\text{-}tilt}\ (A/I)$.
\end{proposition}

Next, we focus on the Schur algebra $S(n,r)$. We give two useful reduction theorems which will allow us to simplify the general problems to the cases with small $n$ and $r$. We point out that $S(n,r)$ with $n>r$ is always Morita equivalent to $S(r,r)$.

\begin{lemma}\label{N>n}
If $S(n,r)$ is $\tau$-tilting infinite, then so is $S(N,r)$ for any $N>n$.
\end{lemma}
\begin{proof}
Let $S$ be the basic algebra of $S(N,r)$. For each $\lambda\in\Omega(n,r)$, we define $e_\lambda$ to be the projector to $Y^\lambda$ and take the sum $e:=\sum e_\lambda$ over all partitions in $\Omega(n,r)$. Then, the idempotent truncation is
\begin{center}
$eSe=e\mathsf{End}_{\mathbb{F}G_r}\left ( \underset{\lambda\in\Omega(N,r)}{\bigoplus} Y^\lambda \right )e=\mathsf{End}_{\mathbb{F}G_r}\left ( \underset{\lambda\in\Omega(n,r)}{\bigoplus} Y^\lambda \right )$.
\end{center}
This implies that the basic algebra of $S(n,r)$ is an idempotent truncation of $S(N,r)$ for any $N>n$. Thus, the statement follows  from Proposition \ref{quotient and idempotent}.
\end{proof}

We recall that the coordinate function $c_{ij}: \mathsf{GL}_n(\mathbb{F})\rightarrow \mathbb{F}$ is defined by $c_{ij}(g)=g_{ij}$ for all $g=[g_{ij}]\in \mathsf{GL}_n(\mathbb{F})$, where $i,j\in \left \{ 1,2,\dots,n \right \}$. Then, we denote by $A(n,r)$ the coalgebra generated by the homogeneous polynomials of total degree $r$ in $c_{ij}$. In fact, the Schur algebra $S(n,r)$ is just the dual of $A(n,r)$.

\begin{lemma}\label{n+r}
If $S(n,r)$ is $\tau$-tilting infinite, then so is $S(n,n+r)$.
\end{lemma}
\begin{proof}
It has been proved in \cite{Erdmann-finite} that $S(n,r)$ is a quotient of $S(n,n+r)$. For convenience, we recall the proof as follows. Let $I=I(n,n+r)$ be the set of maps
\begin{center}
$\alpha :\left \{ 1,2,\dots,n+r \right \}\rightarrow \left \{ 1,2,\dots,n \right \}$
\end{center}
with right $G_{n+r}$-action. Then, $S(n,n+r)$ has a basis $\left \{ \xi _{\alpha, \beta}\mid (\alpha, \beta)\in (I\times I)/G_{n+r} \right \}$: the dual basis of $c_{\alpha(1)\beta(1)}c_{\alpha(2)\beta(2)}\dots c_{\alpha(n+r)\beta(n+r)}\in A(n,n+r)$. Then, the elements $\xi _\alpha:=\xi _{\alpha, \alpha}$ form a set of orthogonal idempotents for $S(n,n+r)$ whose sum is the identity. Note that $\Omega(n, n+r)\subseteq I$ is the set of representatives of $G_{n+r}$-orbits. Let $e=\sum \xi_\lambda$ be the idempotent of $S(n,n+r)$, where the sum is taken over all $\lambda\in \Omega(n, n+r)$ such that $\lambda_n=0$. Then, by using $\mathsf{det}\ (c_{ij})$, we may prove
\begin{center}
$S(n, n+r)/S(n, n+r)eS(n, n+r)\simeq S(n,r)$.
\end{center}
Therefore, the statement follows from Proposition \ref{quotient and idempotent}.
\end{proof}

\subsection{Strategy on $\tau$-tilting infinite Schur algebras}
Let $A:=\mathbb{F}Q/I$ be an algebra presented by a quiver $Q$ and an admissible ideal $I$. We call $Q$ a \emph{$\tau$-tilting infinite quiver} if $A/\mathsf{rad}^2A$ is $\tau$-tilting infinite. For example, the Kronecker quiver $Q:\xymatrix@C=1cm{\circ\ar@<0.5ex>[r]^{\ }\ar@<-0.5ex>[r]_{\ }&\circ}$ is a $\tau$-tilting infinite quiver. Then, the following lemma provides us with three $\tau$-tilting infinite quivers.

\begin{lemma}\label{tau-tilting infinite quiver}
The following quivers $\mathsf{Q}_1, \mathsf{Q}_2$ and $\mathsf{Q}_3$ are $\tau$-tilting infinite quivers.
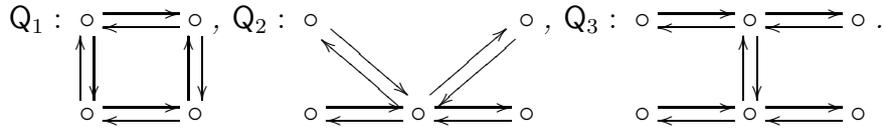
\begin{center}
$\mathsf{Q}_1: \xymatrix@C=1cm{ \circ \ar@<0.5ex>[d]^{}\ar@<0.5ex>[r]^{}&\circ  \ar@<0.5ex>[l]^{}\ar@<0.5ex>[d]^{}\\ \circ  \ar@<0.5ex>[u]^{}\ar@<0.5ex>[r]^{ }&\circ  \ar@<0.5ex>[u]^{ }\ar@<0.5ex>[l]^{}}$, $\mathsf{Q}_2: \xymatrix@C=1cm{ \circ \ar@<0.5ex>[dr]^{} & &\circ  \ar@<0.5ex>[dl]^{}\\ \circ \ar@<0.5ex>[r]^{} &\circ  \ar@<0.5ex>[ur]^{}\ar@<0.5ex>[ul]^{}\ar@<0.5ex>[l]^{}\ar@<0.5ex>[r]^{ }&\circ \ar@<0.5ex>[l]^{}}$, $\mathsf{Q}_3: \xymatrix@C=1cm{ \circ \ar@<0.5ex>[r]^{} &\circ \ar@<0.5ex>[d]^{}\ar@<0.5ex>[r]^{}\ar@<0.5ex>[l]^{}&\circ  \ar@<0.5ex>[l]^{}\\ \circ \ar@<0.5ex>[r]^{} &\circ  \ar@<0.5ex>[u]^{}\ar@<0.5ex>[l]^{}\ar@<0.5ex>[r]^{ }&\circ \ar@<0.5ex>[l]^{}}$.
\end{center}
\end{lemma}
\begin{proof}
We look at the following subquivers.
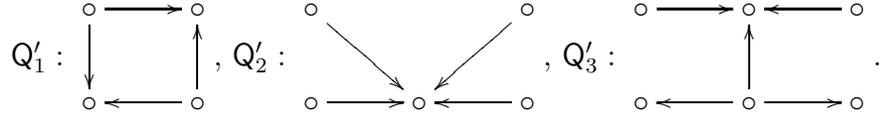
\begin{center}
$\mathsf{Q}_1': \vcenter{\xymatrix@C=1cm{ \circ \ar[d]\ar[r]&\circ \\ \circ&\circ  \ar[u]\ar[l]}}$, $\mathsf{Q}_2': \vcenter{\xymatrix@C=1cm{ \circ \ar[dr]& &\circ  \ar[dl]\\ \circ \ar[r]&\circ&\circ \ar[l]}}$, $\mathsf{Q}_3': \vcenter{\xymatrix@C=1cm{ \circ \ar[r] &\circ &\circ  \ar[l]\\ \circ &\circ  \ar[u]\ar[l]\ar[r]&\circ }}$.
\end{center}
Since the path algebra of $\mathsf{Q}_i'$ for $i=1,2,3$ is a quotient algebra of $A=\mathbb{F}Q/I$ if $Q=\mathsf{Q}_i$, and all of these path algebras are $\tau$-tilting infinite as mentioned in \cite{W-simply}, we conclude that $A/\mathsf{rad}^2A$ is $\tau$-tilting infinite if $Q=\mathsf{Q}_1,\mathsf{Q}_2,\mathsf{Q}_3$ by Proposition \ref{quotient and idempotent}.
\end{proof}

We remark that Adachi \cite{Ada-rad-square-0} (and Aoki \cite{Aoki}) provided a handy criteria for the $\tau$-tilting finiteness of radical square zero algebras, that is, for any algebra $A$, the quotient $A/\mathsf{rad}^2 A$ is $\tau$-tilting finite if and only if every single subquiver of the separated quiver for $A/\mathsf{rad}^2 A$ is a disjoint union of Dynkin quivers. This also gives a proof of Lemma \ref{tau-tilting infinite quiver}.

We mention that in order to show that $S(n,r)$ is $\tau$-tilting infinite, it suffices to find a block algebra of $S(n,r)$ which is Morita equivalent to $\mathbb{F}Q/I$ with a $\tau$-tilting infinite subquiver in $Q$. Then, the advantage is that it is not necessary to find the explicit relations in $I$. As we mentioned in the introduction, except for the cases in $(\star)$, $S(n,r)$ is $\tau$-tilting infinite if and only if it contains one of $\mathsf{Q}_1, \mathsf{Q}_2$ and $\mathsf{Q}_3$ as a subquiver.

To find a $\tau$-tilting infinite subquiver in $S(2,r)$, it is worth mentioning Erdmann and Henke's method \cite{EH-method}. Let $\lambda=(\lambda_1,\lambda_2)$ and $\mu=(\mu_1,\mu_2)$ be two partitions of $r$, we define two non-negative integers $s:=\lambda_1-\lambda_2$ and $t:=\mu_1-\mu_2$. We denote by $v^s$ the vertex in the quiver of $S(2,r)$ corresponding to the Young module $Y^{(\lambda_1,\lambda_2)}$ with $s=\lambda_1-\lambda_2$. Let $n(v^s,v^t)$ be the number of arrows from $v^s$ to $v^t$. Then, it is shown in \cite[Theorem 3.1]{EH-method} that $n(v^s,v^t)=n(v^t,v^s)$ and $n(v^s,v^t)$ is either 0 or 1. We have the following recursive algorithm for computing $n(v^s,v^t)$.
\begin{lemma}{\rm(\cite[Proposition 3.1]{EH-method})}\label{EH-method}
Suppose that $p$ is a prime number and $s>t$. Let $s=s_0+ps'$ and $t=t_0+pt'$ with $0\leqslant s_0,t_0\leqslant p-1$ and $s',t'\geqslant 0$.
\begin{description}\setlength{\itemsep}{-3pt}
  \item[(1)] If $p=2$, then
\begin{center}
$n(v^s,v^t)=\left\{\begin{aligned}
&n(v^{s'},v^{t'}) &\text{ if }&s_0=t_0=1\ \text{or}\ s_0=t_0=0\ \text{and}\ s'\equiv t'\ \rm{mod}\ 2,\\
&1 &\text{if}\ &s_0=t_0=0,t'+1=s'\not\equiv 0\ \rm{mod}\ 2,\\
&0 &\text{ot}&\text{herwise.}
\end{aligned}\right.$
\end{center}
  \item[(2)] If $p>2$, then
\begin{center}
$n(v^s,v^t)=\left\{\begin{aligned}
&n(v^{s'},v^{t'}) &\text{ if }&s_0=t_0,\\
&1 &\text{if}\ &s_0+t_0=p-2,t'+1=s'\not\equiv 0\ {\rm{mod}}\ p,\\
&0 &\text{ot}&\text{herwise.}
\end{aligned}\right.$
\end{center}
\end{description}
\end{lemma}

\section{Representation-finite and tame Schur algebras}\label{section-finite-tame}
In this section, we show that all tame Schur algebras are $\tau$-tilting finite. We first recall the complete classification of the representation type of Schur algebras. Note that some semi-simple cases are contained in the representation-finite cases. We may distinguish the semi-simple cases following \cite{DN-semisimple}. Namely, the Schur algebra $S(n,r)$ is semi-simple if and only if $p=0$ or $p>r$ or $p=2, n=2, r=3$.

\begin{proposition}{\rm{(\cite{Erdmann-finite, DEMN-tame schur}, Modified\footnote{The Schur algebra $S(2,11)$ over $p=2$ was proved to be wild in \cite{DEMN-tame schur}, but this is an error caused by miscalculating the number of simples in each of the blocks of $S(2,11)$. After correcting the number, $S(2,11)$ becomes tame because we can apply \cite[Theorem 13]{EH-two-blocks} to show that $S(2,9)$ and $S(2,11)$ have the same representation type. Thanks to Prof. Erdmann for explaining this.})}}\label{summary}
Let $p$ be the characteristic of $\mathbb{F}$. Then, the Schur algebra $S(n,r)$ is representation-finite if and only if $p=2, n=2, r=5,7$ or $p\geqslant 2, n=2, r<p^2$ or $p\geqslant 2, n\geqslant 3, r<2p$; tame if and only if $p=2, n=2, r=4,9,11$ or $p=3, n=2, r=9, 10, 11$ or $p=3, n=3, r=7,8$. Otherwise, $S(n,r)$ is wild.
\end{proposition}

\subsection{Representation-finite blocks}\label{subsection3.1}
Erdmann \cite[Proposition 4.1]{Erdmann-finite} showed that each block $A$ of a representation-finite Schur algebra $S(n,r)$ is Morita equivalent to $\mathcal{A}_m:=\mathbb{F}Q/I$ for some $m\in \mathbb{N}$, which is defined by the following quiver with relations.
\begin{equation}
\begin{aligned}
\ &\ \ \ Q: \xymatrix@C=1cm@R=0.3cm{1\ar@<0.5ex>[r]^{\alpha_1}&2\ar@<0.5ex>[l]^{\beta_1}\ar@<0.5ex>[r]^{\alpha_2}&\cdots \ar@<0.5ex>[l]^{\beta_2}\ar@<0.5ex>[r]^{\alpha_{m-2}}&m-1\ar@<0.5ex>[l]^{\beta_{m-2}}\ar@<0.5ex>[r]^{\ \ \alpha_{m-1}}&m\ar@<0.5ex>[l]^{\ \ \beta_{m-1}}}, \\ I:& \left \langle \alpha_1\beta_1,\alpha_i\alpha_{i+1},\beta_{i+1}\beta_i, \beta_i\alpha_i-\alpha_{i+1}\beta_{i+1} \mid 1\leqslant i\leqslant m-2 \right \rangle.
\end{aligned}
\end{equation}
Three years later after \cite{Erdmann-finite}, Donkin and Reiten \cite[Theorem 2.1]{DR-finite blocks} generalized this result to an arbitrary Schur algebra $S(n,r)$, that is, each representation-finite block of Schur algebras is Morita equivalent to $\mathcal{A}_m$ for some $m\in \mathbb{N}$.

We would like to determine the number of pairwise non-isomorphic basic support $\tau$-tilting modules for a representation-finite block of Schur algebras.
\begin{theorem}\label{result-finite}
Let $\mathcal{A}_m$ be the algebra defined above. Then, $\#\mathsf{s\tau\text{-}tilt}\ \mathcal{A}_m=\binom{2m}{m}$.
\end{theorem}
\begin{proof}
Let $\Lambda_m$ be the Brauer tree algebra whose Brauer tree is a straight line having $m+1$ vertices and without exceptional vertex. Then, it is easy to check that $\mathcal{A}_m$ is a quotient algebra of $\Lambda_m$ modulo the two-sided ideal generated by $\alpha_1\beta_1$. Since $\alpha_1\beta_1$ is a central element of $\Lambda_m$ and $\#\mathsf{s\tau\text{-}tilt}\ \Lambda_m=\binom{2m}{m}$ has been determined in \cite[Theorem 5.6]{Aoki}, we get the statement following Proposition \ref{center}.
\end{proof}

\subsection{Tame Schur algebras}\label{subsec-3.2}
The block algebras of tame Schur algebras are well-studied in \cite{Erdmann-finite} and \cite{DEMN-tame schur}. In this subsection, we recall these constructions and show that tame Schur algebras are $\tau$-tilting finite. We recall the following bound quiver algebras constructed in \cite{Erdmann-finite}, where the tameness for them is given in \cite[5.5, 5.6, 5.7]{DEMN-tame schur}.
\begin{itemize}\setlength{\itemsep}{-3pt}
\item Let $\mathcal{D}_3:=\mathbb{F}Q/I$ be the special biserial algebra given by
\begin{equation}
Q:\xymatrix@C=1cm{1\ar@<0.5ex>[r]^{\alpha_1}&2\ar@<0.5ex>[l]^{\beta_1}\ar@<0.5ex>[r]^{\alpha_2}&3\ar@<0.5ex>[l]^{\beta_2}}\ \text{and}\ I: \left \langle \begin{matrix}
\alpha_1\beta_1,\beta_2\alpha_2, \alpha_1\alpha_2\beta_2,\alpha_2\beta_2\beta_1
\end{matrix}\right \rangle.
\end{equation}

\item Let $\mathcal{D}_4:=\mathbb{F}Q/I$ be the bound quiver algebra given by
\begin{equation}
Q: \xymatrix@C=1cm@R=0.8cm{1\ar@<0.5ex>[r]^{\alpha_1}&3\ar@<0.5ex>[l]^{\beta_1}\ar@<0.5ex>[d]^{\beta_2}\ar@<0.5ex>[r]^{\beta_3}&4\ar@<0.5ex>[l]^{\alpha_3}\\
&2\ar@<0.5ex>[u]^{\alpha_2}&}\ \text{and}\ I:  \left \langle \begin{matrix}
\alpha_1\beta_1,\alpha_2\beta_2,\alpha_3\beta_1,\alpha_3\beta_2,\alpha_1\beta_3,\alpha_2\beta_3,\\
\alpha_1\beta_2\alpha_2,\beta_2\alpha_2\beta_1,\beta_2\alpha_2-\beta_3\alpha_3
\end{matrix}\right \rangle.
\end{equation}

\item Let $\mathcal{R}_4:=\mathbb{F}Q/I$ be the bound quiver algebra given by
\begin{equation}
Q:\xymatrix@C=1cm{\circ\ar@<0.5ex>[r]^{\alpha_1}&\circ\ar@<0.5ex>[l]^{\beta_1}\ar@<0.5ex>[r]^{\alpha_2}&\circ\ar@<0.5ex>[l]^{\beta_2}\ar@<0.5ex>[r]^{\alpha_3}&\circ\ar@<0.5ex>[l]^{\beta_3}}\ \text{and}\ I: \left \langle\begin{matrix}
\alpha_1\beta_1,\alpha_1\alpha_2, \beta_2\beta_1,\\
\alpha_2\beta_2-\beta_1\alpha_1, \alpha_3\beta_3-\beta_2\alpha_2
\end{matrix}\right \rangle.
\end{equation}

\item Let $\mathcal{H}_4:=\mathbb{F}Q/I$ be the bound quiver algebra given by
\begin{equation}
Q: \xymatrix@C=1cm@R=0.8cm{\circ\ar@<0.5ex>[r]^{\alpha_1}&\circ\ar@<0.5ex>[l]^{\beta_1}\ar@<0.5ex>[d]^{\beta_2}\ar@<0.5ex>[r]^{\alpha_3}&\circ\ar@<0.5ex>[l]^{\beta_3}\\
&\circ\ar@<0.5ex>[u]^{\alpha_2}&}\ \text{and}\ I:  \left \langle \begin{matrix}
\alpha_1\beta_1,\alpha_1\beta_2,\alpha_2\beta_1,\alpha_2\beta_2,\alpha_1\alpha_3,\\
\beta_3\beta_1, \alpha_3\beta_3-\beta_1\alpha_1-\beta_2\alpha_2
\end{matrix}\right \rangle.
\end{equation}
\end{itemize}
We remark that $\mathcal{D}_3$, $\mathcal{D}_4$, $\mathcal{R}_4$ and $\mathcal{H}_4$ are also tame blocks of some wild Schur algebras.

\begin{lemma}\label{method of D_4}
The tame algebras $\mathcal{D}_3$, $\mathcal{D}_4$, $\mathcal{R}_4$ and $\mathcal{H}_4$ are $\tau$-tilting finite. Moreover,
\begin{center}
\renewcommand\arraystretch{1.2}
\begin{tabular}{|c|c|c|c|c|c|}
\hline
$A$&$\mathcal{D}_3$&$\mathcal{D}_4$&$\mathcal{R}_4$&$\mathcal{H}_4$\\ \hline
$\#\mathsf{s\tau\text{-}tilt}\ A$&28&114&88&96 \\ \hline
\end{tabular}.
\end{center}
\end{lemma}
\begin{proof}
We often use Proposition \ref{center} to reduce the direct calculation of left mutations.

(1) Since $\alpha_2\beta_2$ and $\beta_2\beta_1\alpha_1\alpha_2$ are non-trivial central elements\footnote[1]{We use \emph{GAP: a System for Computational Discrete Algebra} to compute the central elements.} of $\mathcal{D}_3$, we may define
\begin{center}
$\widetilde{\mathcal{D}_3}:=\mathcal{D}_3/<\alpha_2\beta_2,\beta_2\beta_1\alpha_1\alpha_2>$
\end{center}
so that $\#\mathsf{s\tau\text{-}tilt}\ \widetilde{\mathcal{D}_3}=\#\mathsf{s\tau\text{-}tilt}\ \mathcal{D}_3$. Then, we determine the number $\#\mathsf{s\tau\text{-}tilt}\ \widetilde{\mathcal{D}_3}$ by calculating the left mutations starting with $\widetilde{\mathcal{D}_3}$. In fact, this is equivalent to finding the Hasse quiver $\mathcal{H}(\mathsf{s\tau\text{-}tilt}\ \widetilde{\mathcal{D}_3})$. We recall that the indecomposable projective $\widetilde{\mathcal{D}_3}$-modules are
\begin{center}
$P_1=\begin{smallmatrix}
1\\
2\\
3
\end{smallmatrix},
P_2=\begin{smallmatrix}
&2&\\
1& &3\\
2\\
3
\end{smallmatrix},
P_3=\begin{smallmatrix}
3\\
2\\
1\\
2
\end{smallmatrix}$.
\end{center}
Starting with the unique maximal $\tau$-tilting module $\widetilde{\mathcal{D}_3}\simeq P_1\oplus P_2\oplus P_3$, we take an exact sequence with a minimal left $\mathsf{add}(P_2\oplus P_3)$-approximation $f_1$ of $P_1$:
\begin{center}
$\substack{1\\2\\3}\overset{f_1}{\longrightarrow}\substack{\ \\1\\2\\3}\substack{2\\\ \\\ \\ \  }\substack{\ \\3\\\ \\ \ }\longrightarrow \mathsf{coker}\ f_1\longrightarrow 0$.
\end{center}
Then, $\mathsf{coker}\ f_1=\substack{2\\3}$ and $\mu_{P_1}^-(\widetilde{\mathcal{D}_3})= \substack{2\\3}\oplus P_2\oplus P_3$. Next, we take an exact sequence with a minimal left $\mathsf{add}(\substack{2\\3}\oplus P_3)$-approximation $f_2$ of $P_2$:
\begin{center}
$P_2\overset{f_2}{\longrightarrow}\substack{2\\3}\oplus P_3\longrightarrow \mathsf{coker}\ f_2\longrightarrow 0$,
\end{center}
so that $\mathsf{coker}\ f_2=\substack{3\\2}$ and $\mu_{P_2}^-(\mu_{P_1}^-(\widetilde{\mathcal{D}_3}))= \substack{2\\3}\oplus \substack{3\\2}\oplus P_3$. Similarly, $\mu_{P_3}^-(\mu_{P_2}^-(\mu_{P_1}^-(\widetilde{\mathcal{D}_3})))= \substack{2\\3}\oplus \substack{3\\2}$. Then, by the calculation in Example \ref{example-simplest}, we have
\begin{center}
\begin{tikzpicture}[shorten >=1pt, auto, node distance=0cm,
   node_style/.style={font=},
   edge_style/.style={draw=black}]
\node[node_style] (v) at (0,0) {$\substack{2\\3}\oplus \substack{3\\2}$};
\node[node_style] (v1) at (2.5,0.7) {$\substack{3}\oplus \substack{3\\2}$};
\node[node_style] (v12) at (5,0.7) {$\substack{3}$};
\node[node_style] (v2) at (2.5,-0.7) {$\substack{2\\3}\oplus \substack{2}$};
\node[node_style] (v21) at (5,-0.7) {$\substack{2}$};
\node[node_style] (v0) at (7,0) {$\substack{0}$};
\draw[->]  (v) edge node{}(v1);
\draw[->]  (v1) edge node{\ }(v12);
\draw[->]  (v12) edge node{\ }(v0);
\draw[->]  (v) edge node{}(v2);
\draw[->]  (v2) edge node{}(v21);
\draw[->]  (v21) edge node{}(v0);
\end{tikzpicture}.
\end{center}

In this way, one can calculate all possible left mutation sequences starting with $\widetilde{\mathcal{D}_3}$ and ending at $0$, so that the Hasse quiver $\mathcal{H}(\mathsf{s\tau\text{-}tilt}\ \widetilde{\mathcal{D}_3})$ is as follows.
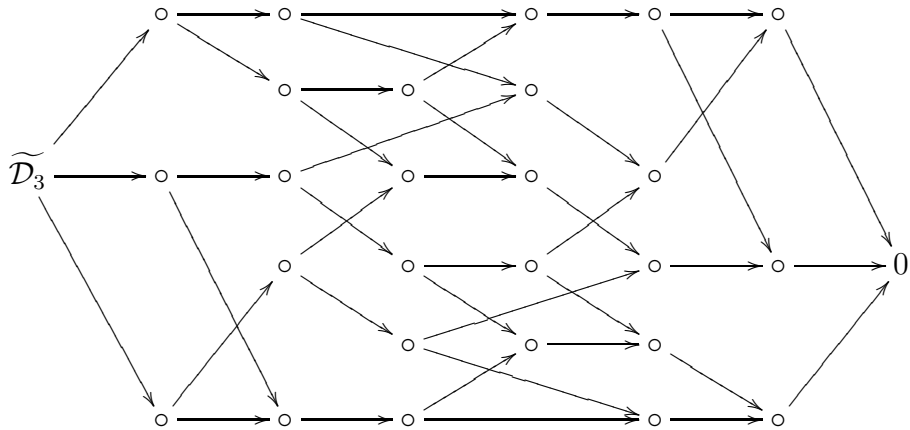
\begin{center}
$\vcenter{\xymatrix@C=1.2cm@R=0.6cm{&\circ\ar[dr]\ar[r]&\circ\ar[rr]\ar[drr]&&\circ\ar[r]&\circ\ar[dddr]\ar[r]&\circ\ar[dddr]&\\
&&\circ\ar[r]\ar[dr]&\circ\ar[ur]\ar[dr]&\circ\ar[dr]& &&\\
\widetilde{\mathcal{D}_3}\ar[r]\ar[uur]\ar[dddr]&\circ\ar[r]\ar[dddr]&\circ\ar[dr]\ar[urr]&\circ\ar[r]&\circ\ar[dr]&\circ\ar[uur]&&\\
&&\circ\ar[dr]\ar[ur]&\circ\ar[r]\ar[dr]&\circ\ar[dr]\ar[ur]&\circ\ar[r]&\circ\ar[r]&0\\
&&&\circ\ar[drr]\ar[urr]&\circ\ar[r]&\circ\ar[dr]&&\\
&\circ\ar[uur]\ar[r]&\circ\ar[r]&\circ\ar[rr]\ar[ur]&&\circ\ar[r]&\circ\ar[uur]&}}$.
\end{center}
Hence, we deduce that $\mathcal{D}_3$ is $\tau$-tilting finite and $\#\mathsf{s\tau\text{-}tilt}\ \mathcal{D}_3=28$.

(2) Since $\beta_1\alpha_1$ and $\beta_2\alpha_2+\beta_3\alpha_3$ are non-trivial central elements of $\mathcal{R}_4$, we define
\begin{center}
$\widetilde{\mathcal{R}_4}:=\mathcal{R}_4/<\beta_1\alpha_1,\beta_2\alpha_2+\beta_3\alpha_3>$
\end{center}
and then, $\#\mathsf{s\tau\text{-}tilt}\ \widetilde{\mathcal{R}_4}=\#\mathsf{s\tau\text{-}tilt}\ \mathcal{R}_4$. Instead of direct calculation, we point out that $\widetilde{\mathcal{R}_4}$ is a representation-finite string algebra and $\mathcal{H}(\mathsf{s\tau\text{-}tilt}\ \widetilde{\mathcal{R}_4})$ can be constructed by the String Applet \cite{G-string applet}. Thus, we deduce that $\#\mathsf{s\tau\text{-}tilt}\ \mathcal{R}_4=88$.

(3) Since the non-trivial central elements of $\mathcal{D}_4$ are $\beta_2\alpha_2$, $\alpha_3\beta_3$ and $\alpha_2\beta_1\alpha_1\beta_2$, we define
\begin{center}
$\widetilde{\mathcal{D}_4}:=\mathcal{D}_4/<\beta_2\alpha_2,\alpha_3\beta_3,\alpha_2\beta_1\alpha_1\beta_2>$
\end{center}
so that $\#\mathsf{s\tau\text{-}tilt}\ \widetilde{\mathcal{D}_4}=\#\mathsf{s\tau\text{-}tilt}\ \mathcal{D}_4$. We explain our strategy for computing $\#\mathsf{s\tau\text{-}tilt}\ \widetilde{\mathcal{D}_4}$.

We denote by $a_s(\widetilde{\mathcal{D}_4})$ the number of pairwise non-isomorphic basic support $\tau$-tilting $\widetilde{\mathcal{D}_4}$-modules $M$ with support-rank $s$ (i.e., $|M|=s$) for $0\leqslant s\leqslant 4$. Then, we have
\begin{center}
$\#\mathsf{s\tau\text{-}tilt}\ \widetilde{\mathcal{D}_4}=\sum\limits_{s=0}^{n}a_s(\widetilde{\mathcal{D}_4})$.
\end{center}
Since the support $\tau$-tilting $\widetilde{\mathcal{D}_4}$-module with support-rank $0$ is unique, we have $a_0(\widetilde{\mathcal{D}_4})=1$. Since each simple $\widetilde{\mathcal{D}_4}$-module $S_i$ is a $\widetilde{\mathcal{D}_4}/<1-e_i>$-module and $\widetilde{\mathcal{D}_4}/<1-e_i>\simeq \mathbb{F}$, we observe that the support $\tau$-tilting $\widetilde{\mathcal{D}_4}$-modules with support-rank $1$ are exactly the simple $\widetilde{\mathcal{D}_4}$-modules. Then, $a_1(\widetilde{\mathcal{D}_4})=|\widetilde{\mathcal{D}_4}|=4$.

Let $M$ be a support $\tau$-tilting $\widetilde{\mathcal{D}_4}$-module with support-rank 2, and with supports $e_i$ and $e_j$ ($i\neq j$). Then, $M$ becomes a $\tau$-tilting $\widetilde{\mathcal{D}_4}/J$-module with $J:=<1-e_i-e_j>$. We denote by $\mathsf{b}_{i,j}$ the number of $\tau$-tilting $\widetilde{\mathcal{D}_4}/J$-modules. For example, if $(i,j)=(1,3)$, then
\begin{center}
$\widetilde{\mathcal{D}_4}/J=\mathbb{F}\left ( \xymatrix@C=1cm@R=0.8cm{1\ar@<0.5ex>[r]^{\alpha_1}&3\ar@<0.5ex>[l]^{\beta_1}} \right )/\left \langle \alpha_1\beta_1\right \rangle$.
\end{center}
Note that $\beta_1\alpha_1$ is a central element of $\widetilde{\mathcal{D}_4}/J$, we may apply Proposition \ref{center} and Example \ref{example-simplest} to show that $\mathcal{H}(\mathsf{s\tau\text{-}tilt}\ \widetilde{\mathcal{D}_4}/J)$ is displayed below,
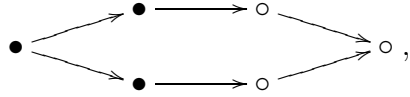
\begin{center}
$\vcenter{\xymatrix@C=1.2cm@R=0.1cm{& \bullet\ar[r]^{} &\circ\ar[dr]^{}\\\bullet\ar[dr]^{}\ar[ur]^{}&&&\circ \\&\bullet\ar[r]^{}&\circ\ar[ur]^{}&}}$,
\end{center}
where we denote by $\bullet$ $\tau$-tilting $\widetilde{\mathcal{D}_4}/J$-modules and by $\circ$ other support $\tau$-tilting (but not $\tau$-tilting) $\widetilde{\mathcal{D}_4}/J$-modules. Hence, we deduce that $\mathsf{b}_{1,3}=3$. Similarly, we have
\begin{center}
\begin{tabular}{c|cccccccccccc}
$(i,j)$&$(1,2)$&$(1,4)$&$(2,3)$&$(2,4)$&$(3,4)$ \\ \hline
$\mathsf{b}_{i,j}$&1&1&3&1&3\\
\end{tabular}
\end{center}
This implies that $a_2(\widetilde{\mathcal{D}_4})=12$.

Let $N$ be a support $\tau$-tilting $\widetilde{\mathcal{D}_4}$-module with support-rank 3. Then, $N$ becomes a $\tau$-tilting $\widetilde{\mathcal{D}_4}/L_j$-module with $L_j:=<e_j>$, where $e_j$ is the only one non-zero primitive idempotent satisfying $Ne_j=0$. We denote by $\mathsf{d}_j$ the number of $\tau$-tilting $\widetilde{\mathcal{D}_4}/L_j$-modules. For example, if $j=4$, then
\begin{center}
$\widetilde{\mathcal{D}_4}/L_4=\mathbb{F}\left ( \xymatrix@C=1cm@R=0.8cm{1\ar@<0.5ex>[r]^{\alpha_1}&3\ar@<0.5ex>[l]^{\beta_1}\ar@<0.5ex>[r]^{\beta_2}&2\ar@<0.5ex>[l]^{\alpha_2}} \right )/\left \langle \alpha_1\beta_1,\alpha_2\beta_2, \beta_2\alpha_2, \alpha_2\beta_1\alpha_1\beta_2 \right \rangle$.
\end{center}
Since $\widetilde{\mathcal{D}_4}/L_4$ is isomorphic to $\widetilde{\mathcal{D}_3}$, $\mathsf{d}_4=17$ by the calculation before. If $j=2$, then
\begin{center}
$\widetilde{\mathcal{D}_4}/L_2=\mathbb{F}\left ( \xymatrix@C=1cm@R=0.8cm{1\ar@<0.5ex>[r]^{\alpha_1}&3\ar@<0.5ex>[l]^{\beta_1}\ar@<0.5ex>[r]^{\beta_3}&4\ar@<0.5ex>[l]^{\alpha_3}}\right )/\left \langle \alpha_1\beta_1,\alpha_3\beta_1,\alpha_1\beta_3,\beta_3\alpha_3, \alpha_3\beta_3\right \rangle$,
\end{center}
and the Hasse quiver $\mathcal{H}(\mathsf{s\tau\text{-}tilt}\ \widetilde{\mathcal{D}_4}/L_2)$ is shown as follows by direct calculation.
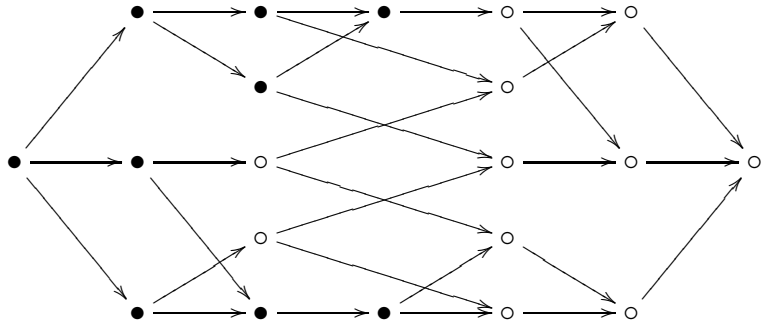
\begin{center}
$\xymatrix@C=1.2cm@R=0.6cm{&\bullet\ar[dr]\ar[r]&\bullet\ar[r]\ar[drr]&\bullet\ar[r]&\circ\ar[r]\ar[ddr]&\circ\ar[ddr]&\\
&&\bullet\ar[drr]\ar[ur]& &\circ\ar[ur] & &\\
\bullet\ar[r]\ar[uur]\ar[ddr]&\bullet\ar[r]\ar[ddr]&\circ\ar[drr]\ar[urr] &&\circ\ar[r]&\circ\ar[r]&\circ\\
&&\circ\ar[drr]\ar[urr]& &\circ\ar[dr]& &\\
&\bullet\ar[ur]\ar[r]&\bullet\ar[r]&\bullet\ar[r]\ar[ur]&\circ\ar[r]&\circ\ar[uur]&}$
\end{center}
We deduce that $\mathsf{d}_2=9$. Similarly, we have $\mathsf{d}_1=9$ and $\mathsf{d}_3=1$. Therefore, $a_3(\widetilde{\mathcal{D}_4})=36$.

We may compute $a_4(\widetilde{\mathcal{D}_4})$ by hand, because the support $\tau$-tilting $\widetilde{\mathcal{D}_4}$-modules with support-rank 4 are just $\tau$-tilting $\widetilde{\mathcal{D}_4}$-modules, which can be obtained by the left mutations starting with $\widetilde{\mathcal{D}_4}$. See Appendix A for a complete list of $\tau$-tilting $\widetilde{\mathcal{D}_4}$-modules and one may easily construct the part of $\mathcal{H}(\mathsf{s\tau\text{-}tilt}\ \widetilde{\mathcal{D}_4})$ consisting of all $\tau$-tilting $\widetilde{\mathcal{D}_4}$-modules. Then, $a_4(\widetilde{\mathcal{D}_4})=61$ and we conclude that $\#\mathsf{s\tau\text{-}tilt}\ \widetilde{\mathcal{D}_4}=1+4+12+36+61=114$.

(4) Since the non-trivial central elements of $\mathcal{H}_4$ are $\beta_1\alpha_1$, $\beta_2\alpha_2+\beta_3\alpha_3$ and $\beta_3\beta_2\alpha_2\alpha_3$, we define
\begin{center}
$\widetilde{\mathcal{H}_4}:=\mathcal{H}_4/<\beta_1\alpha_1,\beta_2\alpha_2+\beta_3\alpha_3,\beta_3\beta_2\alpha_2\alpha_3>$
\end{center}
so that $\#\mathsf{s\tau\text{-}tilt}\ \widetilde{\mathcal{H}_4}=\#\mathsf{s\tau\text{-}tilt}\ \mathcal{H}_4$. Similar to the case of $\mathcal{D}_4$, one can compute the left mutations starting with $\widetilde{\mathcal{H}_4}$ to find all $\tau$-tilting $\widetilde{\mathcal{H}_4}$-modules and the number is 47. Then,
\renewcommand\arraystretch{1.2}
\begin{center}
\begin{tabular}{c|ccccc||ccccccc}
$s$&$0$&$1$&$2$&$3$&$4$&$\#\mathsf{s\tau\text{-}tilt}\ \widetilde{\mathcal{H}_4}$  \\ \hline
$a_s(\widetilde{\mathcal{H}_4})$&1&4&12&32&47&96\\
\end{tabular}
\end{center}
Also, the part of $\mathcal{H}(\mathsf{s\tau\text{-}tilt}\ \widetilde{\mathcal{H}_4})$ consisting of all $\tau$-tilting $\widetilde{\mathcal{H}_4}$-modules can be obtained.
\end{proof}

\begin{theorem}\label{result-tame}
If the Schur algebra $S(n,r)$ is tame, it is $\tau$-tilting finite.
\end{theorem}
\begin{proof}
We have proved in Lemma \ref{method of D_4} that $\mathcal{D}_3$, $\mathcal{D}_4$, $\mathcal{R}_4$ and $\mathcal{H}_4$ are $\tau$-tilting finite. Now, it suffices to make clear that these are all the tame blocks of tame Schur algebras. By Proposition \ref{summary}, it is enough to consider $S(2,r)$ for $r=4,9,11$ over $p=2$, $S(2,r)$ for $r=9, 10, 11$ and $S(3,r)$ for $r=7,8$ over $p=3$. We have already shown in Example \ref{s(2,11)} that the basic algebra of $S(2,11)$ over $p=2$ is isomorphic to $\mathcal{D}_3\oplus \mathcal{A}_2\oplus \mathbb{F}$. Then, the basic algebra of other tame Schur algebras can be found in \cite[Section 5]{DEMN-tame schur}. We recall the result in \cite{DEMN-tame schur} as follows.

Let $p=2$, the basic algebra of $S(2,4)$ is isomorphic to $\mathcal{D}_3$ and the basic algebra of $S(2,9)$ is isomorphic to $\mathcal{D}_3\oplus \mathbb{F}\oplus \mathbb{F}$. Let $p=3$, the basic algebra of $S(2,9)$ is isomorphic to $\mathcal{D}_4\oplus \mathbb{F}$, the basic algebra of $S(2,10)$ is isomorphic to $\mathcal{D}_4\oplus \mathbb{F}\oplus \mathbb{F}$ and the basic algebra of $S(2,11)$ is isomorphic to $\mathcal{D}_4\oplus \mathcal{A}_2$; the basic algebra of $S(3,7)$ is isomorphic to $\mathcal{R}_4\oplus \mathcal{A}_2\oplus \mathcal{A}_2$ and the basic algebra of $S(3,8)$ is isomorphic to $\mathcal{R}_4\oplus \mathcal{H}_4\oplus \mathcal{A}_2$.
\end{proof}

\section{Wild Schur Algebras}\label{section-wild}
As we mentioned in the introduction, the cases in $(\star)$ have been distinguished and we will deal with them in the subsection \ref{remaining cases}. Now, we determine the $\tau$-tilting finiteness for wild Schur algebras over various $p$ as follows. In this section, we will use the decomposition matrix $[S^\lambda:D^\mu]$ of $\mathbb{F}G_r$ given in \cite{James-symmetric} without further notice.

\subsection{The characteristic $p=2$}
We assume in this subsection that the characteristic of $\mathbb{F}$ is 2. Then, the $\tau$-tilting finiteness for $S(n,r)$ is shown in Table \ref{p=2} and the proof is divided into several propositions as displayed below. In particular, we claim that the color \textbf{purple} means $\tau$-tilting finite, the color \textbf{red} means $\tau$-tilting infinite, the capital letter \textbf{S} means semi-simple, the capital letter \textbf{F} means representation-finite, the capital letter \textbf{T} means tame and the capital letter \textbf{W} means wild.
\begin{table}
\caption{The $\tau$-tilting finite $S(n,r)$ over $p=2$.}\label{p=2}
\centering
\renewcommand\arraystretch{1.5}
\begin{tabular}{|c|c|c|c|c|c|c|c|c|c|c|c|c|c|c|c|c}
\hline
$r$&1&2&3&4&5&6&7&8&9&10&11&12\\ \hline
$S(2,r)$&\cellcolor{myblue}S&\cellcolor{myblue}F&\cellcolor{myblue}S&\cellcolor{myblue}T&\cellcolor{myblue}F&\cellcolor{myblue}W&\cellcolor{myblue}F&\cellcolor{mygray}W
&\cellcolor{myblue}T&\cellcolor{myred}W&\cellcolor{myblue}T&\cellcolor{myred}W\\ \hline\hline
$r$&13&14&15&16&17&18&19&20&21&22&23&$\cdots$\\ \hline
$S(2,r)$&\cellcolor{myblue}W&\cellcolor{myred}W&\cellcolor{myblue}W&\cellcolor{myred}W&\cellcolor{mygray}W&\cellcolor{myred}W&\cellcolor{mygray}W&
\cellcolor{myred}W&\cellcolor{myred}W&\cellcolor{myred}W&\cellcolor{myred}W&\cellcolor{myred}$\cdots$      \\\hline
\end{tabular}
\begin{tabular}{|c|c|c|c|c|c|c|c|c|c|c|c|c|c|c}
\hline
\diagbox{$n$}{$r$}&1&2&3&4&5&6&7&8&9&10&11&12&13&$\cdots$\\ \hline
3&\cellcolor{myblue}S&\cellcolor{myblue}F&\cellcolor{myblue}F&\cellcolor{mygray}W&\cellcolor{myblue}W&\cellcolor{myred}W&\cellcolor{myred}W&\cellcolor{myred}W&\cellcolor{myred}W&\cellcolor{myred}W&
\cellcolor{myred}W&\cellcolor{myred}W&\cellcolor{myred}W&\cellcolor{myred}$\cdots$\\ \hline
4&\cellcolor{myblue}S&\cellcolor{myblue}F&\cellcolor{myblue}F&\cellcolor{myred}W&\cellcolor{myblue}W&\cellcolor{myred}W&\cellcolor{myred}W&\cellcolor{myred}W&\cellcolor{myred}W&\cellcolor{myred}W&
\cellcolor{myred}W&\cellcolor{myred}W&\cellcolor{myred}W&\cellcolor{myred}$\cdots$\\ \hline
5&\cellcolor{myblue}S&\cellcolor{myblue}F&\cellcolor{myblue}F&\cellcolor{myred}W&\cellcolor{mygray}W&\cellcolor{myred}W&\cellcolor{myred}W&\cellcolor{myred}W&\cellcolor{myred}W&\cellcolor{myred}W&
\cellcolor{myred}W&\cellcolor{myred}W&\cellcolor{myred}W&\cellcolor{myred}$\cdots$\\ \hline
6&\cellcolor{myblue}S&\cellcolor{myblue}F&\cellcolor{myblue}F&\cellcolor{myred}W&\cellcolor{mygray}W&\cellcolor{myred}W&\cellcolor{myred}W&\cellcolor{myred}W&\cellcolor{myred}W&\cellcolor{myred}W&
\cellcolor{myred}W&\cellcolor{myred}W&\cellcolor{myred}W&\cellcolor{myred}$\cdots$\\ \hline
$\vdots$&\cellcolor{myblue}$\vdots$&\cellcolor{myblue}$\vdots$&\cellcolor{myblue}$\vdots$&\cellcolor{myred}$\vdots$&\cellcolor{mygray}$\vdots$&\cellcolor{myred}$\vdots$&
\cellcolor{myred}$\vdots$&\cellcolor{myred}$\vdots$&\cellcolor{myred}$\vdots$&\cellcolor{myred}$\vdots$&\cellcolor{myred}$\vdots$&\cellcolor{myred}
$\vdots$&\cellcolor{myred}$\vdots$&\cellcolor{myred}$\ddots$
\end{tabular}
\end{table}

\begin{proposition}\label{S(2,13)}
Let $p=2$. Then, the wild Schur algebras $S(2,6)$, $S(2,13)$ and $S(2,15)$ are $\tau$-tilting finite.
\end{proposition}
\begin{proof}
We consider the Young modules $Y^\lambda$ for partitions $\lambda$ with at most two parts.

(1) The part of the decomposition matrix $[S^\lambda:D^\mu]$ for the partitions in the principal block of $\mathbb{F}G_{6}$ with at most two parts is
\begin{center}
$\begin{matrix}
(6)\\
(5,1)\\
(4,2)\\
(3^2)
\end{matrix}\begin{pmatrix}
1& \\
1&1\\
1&1&1\\
1&0&1
\end{pmatrix}$,
\end{center}
and the characters of Young $\mathbb{F}G_{6}$-modules are given by Theorem \ref{henke charac} as follows.
\begin{center}
$\begin{aligned}
\mathsf{ch}\ Y^{(6)}&=\chi^{(6)},\\
\mathsf{ch}\ Y^{(5,1)}&=\chi^{(6)}+\chi^{(5,1)},\\
\mathsf{ch}\ Y^{(4,2)}&=\chi^{(5,1)}+\chi^{(4,2)},\\
\mathsf{ch}\ Y^{(3^2)}&=\chi^{(6)}+\chi^{(5,1)}+\chi^{(4,2)}+\chi^{(3^2)}.
\end{aligned}$
\end{center}
Similar to the strategy in the proof of Example \ref{s(2,11)}, we may compute the radical series of Young $\mathbb{F}G_{6}$-modules. Then, we can show that $S_B=\mathsf{End}_{\mathbb{F}G_{6}}(Y^{(6)}\oplus Y^{(5,1)}\oplus Y^{(4,2)}\oplus Y^{(3^2)})$ is isomorphic to $\mathcal{K}_4:=\mathbb{F}Q/I$ with
\begin{equation}
Q:\xymatrix@C=1cm@R=0.3cm{\circ\ar@<0.5ex>[r]^{\alpha_1}&\circ\ar@<0.5ex>[l]^{\beta_1}\ar@<0.5ex>[r]^{\alpha_2}&\circ\ar@<0.5ex>[l]^{\beta_2}\ar@<0.5ex>[r]^{\alpha_3}&\circ\ar@<0.5ex>[l]^{\beta_3}}\ \text{and}\ I: \left \langle \begin{matrix}
\alpha_1\beta_1,\alpha_2\beta_2,\beta_3\alpha_3, \alpha_1\alpha_2\alpha_3, \beta_3\beta_2\beta_1,\\
\beta_1\alpha_1\alpha_2-\alpha_2\alpha_3\beta_3, \beta_2\beta_1\alpha_1-\alpha_3\beta_3\beta_2
\end{matrix}\right \rangle.
\end{equation}
(See \cite[3.5]{DEMN-tame schur} for another method to show this.) Since $\beta_1\alpha_1+\alpha_3\beta_3$ and $\beta_3\beta_2\alpha_2\alpha_3$ are non-trivial central elements of $\mathcal{K}_4$, the number of (pairwise non-isomorphic basic) support $\tau$-tilting modules for $\mathcal{K}_4$ is the same as $\mathcal{K}_4/\left \langle \beta_1\alpha_1,\alpha_3\beta_3, \beta_3\beta_2\alpha_2\alpha_3\right \rangle$ by Proposition \ref{center}. Then, similar to the proof method of Lemma \ref{method of D_4}, we have
\renewcommand\arraystretch{1.2}
\begin{center}
\begin{tabular}{c|ccccc||ccccccc}
$s$&$0$&$1$&$2$&$3$&$4$&$\#\mathsf{s\tau\text{-}tilt}\ \mathcal{K}_4$  \\ \hline
$a_s(\mathcal{K}_4)$&1&4&12&36&83&136\\
\end{tabular}.
\end{center}
This implies that $S_B$ is $\tau$-tilting finite. Hence, the statement follows from the fact that $S_B$ is the basic algebra of $S(2,6)$.

(2) The group algebra $\mathbb{F}G_{13}$ contains two blocks, i.e., the principal block $B_1$ and the block $B_2$ labeled by $2$-core $(2,1)$. The parts of the decomposition matrix $[S^\lambda:D^\mu]$ for the partitions in $B_1$ and $B_2$ with at most two parts are
\begin{center}
$B_1:\begin{matrix}
(13)\\
(11,2)\\
(9,4)\\
(7,6)
\end{matrix}\begin{pmatrix}
1&&&\\
1&1&&\\
1&1&1&\\
1&0&1&1
\end{pmatrix}$, $B_2: \begin{matrix}
(12,1)\\
(10,3)\\
(8,5)
\end{matrix}\begin{pmatrix}
1&&\\
0&1&\\
1&0&1
\end{pmatrix}$.
\end{center}

We may prove that $S_{B_1}$ is isomorphic to $\mathcal{K}_4$, because the characters of Young modules for $\mathbb{F}G_{13}$ are as follows. (One may compare this with the case of $S(2,6)$.)
\begin{center}
$\begin{aligned}
\mathsf{ch}\ Y^{(13)}&=\chi^{(13)},\\
\mathsf{ch}\ Y^{(11,2)}&=\chi^{(13)}+\chi^{(11,2)},\\
\mathsf{ch}\ Y^{(9,4)}&=\chi^{(11,2)}+\chi^{(9,4)},\\
\mathsf{ch}\ Y^{(7,6)}&=\chi^{(13)}+\chi^{(11,2)}+\chi^{(9,4)}+\chi^{(7,6)}.
\end{aligned}$
\end{center}

On the other hand, $S_{B_2}$ is isomorphic to $\mathcal{A}_2\oplus \mathbb{F}$ by \cite[Proposition 4.1]{Erdmann-finite} as we have explained at the start of subsection \ref{subsection3.1}. Therefore, the basic algebra of $S(2,13)$ is $\mathcal{K}_4\oplus \mathcal{A}_2\oplus \mathbb{F}$, which is also $\tau$-tilting finite.

(3) The group algebra $\mathbb{F}G_{15}$ also contains two blocks and the parts of the decomposition matrix $[S^\lambda:D^\mu]$ for the partitions with at most two parts are as follows.
\begin{center}
$\begin{matrix}
(15)\\
(13,2)\\
(11,4)\\
(9,6)
\end{matrix}\begin{pmatrix}
1&&&\\
0&1&&\\
0&0&1&\\
0&1&0&1
\end{pmatrix}$, $\begin{matrix}
(14,1)\\
(12,3)\\
(10,5)\\
(8,7)
\end{matrix}\begin{pmatrix}
1&&\\
1&1&\\
1&1&1\\
1&0&1&1
\end{pmatrix}$.
\end{center}
After computing the characters of Young $\mathbb{F}G_{15}$-modules by Theorem \ref{henke charac}, we deduce that the basic algebra of $S(2,15)$ is isomorphic to $\mathcal{K}_4\oplus \mathcal{A}_2\oplus \mathbb{F}\oplus \mathbb{F}$.
\end{proof}

Now, we look at the case $S(2,8)$. Let $B$ be the principal block of $\mathbb{F}G_{8}$, the part of the decomposition matrix $[S^\lambda:D^\mu]$ for the partitions in $B$ with at most two parts is
\begin{center}
$\begin{matrix}
(8)\\
(7,1)\\
(6,2)\\
(5,3)\\
(4^2)
\end{matrix}\begin{pmatrix}
1& \\
1&1\\
0&1&1\\
0&1&1&1\\
0&1&0&1
\end{pmatrix}$.
\end{center}
On the other hand, Theorem \ref{henke charac} implies that the characters of Young $\mathbb{F}G_{8}$-modules are
\begin{center}
$\begin{aligned}
\mathsf{ch}\ Y^{(8)}&=\chi^{(8)},\\
\mathsf{ch}\ Y^{(7,1)}&=\chi^{(8)}+\chi^{(7,1)},\\
\mathsf{ch}\ Y^{(6,2)}&=\chi^{(8)}+\chi^{(7,1)}+\chi^{(6,2)},\\
\mathsf{ch}\ Y^{(5,3)}&=\chi^{(6,2)}+\chi^{(5,3)},\\
\mathsf{ch}\ Y^{(4^2)}&=\chi^{(8)}+\chi^{(7,1)}+\chi^{(6,2)}+\chi^{(5,3)}+\chi^{(4^2)}.
\end{aligned}$
\end{center}
It is obvious that $Y^{(8)}=D^{(8)}$ and we may find others as follows (I am grateful to Prof. Ariki for showing me the structure of $Y^{(4^2)}$).
\begin{center}
$Y^{(7,1)}=\begin{matrix}
D^{(8)}\\
D^{(7,1)}\\
D^{(8)}
\end{matrix}$, $Y^{(6,2)}=\vcenter{\xymatrix@C=0.01cm@R=0.2cm{&D^{(7,1)}\ar@{-}[dr]\ar@{-}[ddl]&&&D^{(8)}\ar@{-}[ddl] \\
& &D^{(6,2)}\ar@{-}[dr]&&\\
D^{(8)}&&&D^{(7,1)}&}}$, $Y^{(5,3)}=\begin{matrix}
D^{(6,2)}\\
D^{(7,1)}\\
D^{(5,3)}\\
D^{(7,1)}\\
D^{(6,2)}
\end{matrix}$, $Y^{(4^2)}=\vcenter{\xymatrix@C=0.1cm@R=0.3cm{&D^{(7,1)}\ar@{-}[d]\ar@{-}[dr]\ar@{-}@/_1pc/[ddddl]&&D^{(8)}\ar@{-}@/^1pc/[ddddl]\\
&D^{(5,3)}\ar@{-}[d]&D^{(6,2)}\ar@{-}[d]&\\
&D^{(7,1)}\ar@{-}[d]&D^{(7,1)}\ar@{-}[d]&\\
&D^{(6,2)}\ar@{-}[dr]&D^{(5,3)}\ar@{-}[d]&\\
D^{(8)}&&D^{(7,1)}&
}}$.
\end{center}
Note that the dimension of $\mathsf{Hom}_{\mathbb{F}G_{8}}(Y^\lambda,Y^\mu)$ between two Young modules $Y^\lambda,Y^\mu$ is equal to the inner product $(\mathsf{ch}\ Y^\lambda, \mathsf{ch}\ Y^\mu)$. By direct calculation, we conclude that $S_B=\mathsf{End}_{\mathbb{F}G_{8}}(Y^{(8)}\oplus Y^{(7,1)}\oplus Y^{(6,2)}\oplus Y^{(5,3)}\oplus Y^{(4^2)})$ is isomorphic to $\mathcal{L}_5:=\mathbb{F}Q/I$ with
\begin{equation}
\begin{aligned}
\ & \ \ \ \ \ \ \ \ \ \ \ \ \ Q: \xymatrix@C=1cm@R=0.8cm{(5,3)\ar@<0.5ex>[r]^{\alpha_1} &(4^2)\ar@<0.5ex>[r]^{\alpha_2} \ar@<0.5ex>[l]^{\beta_1} \ar@<0.5ex>[d]^{\alpha_4} &(6,2)\ar@<0.5ex>[l]^{\beta_2}\ar@<0.5ex>[r]^{\alpha_3}&(7,1)\ar@<0.5ex>[l]^{\beta_3} \\
&(8)\ar@<0.5ex>[u]^{\beta_4}&} \ \text{and}\\
I:& \left \langle \begin{matrix}
\alpha_1\beta_1,\alpha_1\alpha_4,\beta_3\alpha_3,\beta_2\alpha_2, \beta_4\alpha_4,\beta_4\beta_1, \beta_4\alpha_2\beta_2, \alpha_1\alpha_2\alpha_3,\alpha_2\beta_2\alpha_4,\beta_3\beta_2\beta_1,\\
\beta_1\alpha_1\alpha_2-\alpha_2\alpha_3\beta_3,\beta_2\beta_1\alpha_1-\alpha_3\beta_3\beta_2,\alpha_2\beta_2\beta_1\alpha_1-\beta_1\alpha_1\alpha_2\beta_2
\end{matrix}\right \rangle.
\end{aligned}
\end{equation}
Here, we replace each vertex in the quiver of $S_B$ by the partition $\lambda$ associated with the Young module $Y^\lambda$. We have checked that there are at least 500 pairwise non-isomorphic basic support $\tau$-tilting $\mathcal{L}_5$-modules, so we leave this case to the future.

\begin{proposition}\label{S(2,8)over p=2}
Let $p=2$. Then, the wild Schur algebras $S(2,17)$ and $S(2,19)$ are $\tau$-tilting finite if and only if $S(2,8)$ is $\tau$-tilting finite.
\end{proposition}
\begin{proof}
We show that the basic algebra of $S(2,17)$ is isomorphic to $\mathcal{L}_5\oplus \mathcal{A}_2\oplus \mathbb{F}\oplus \mathbb{F}$. The blocks of $\mathbb{F}G_{17}$ and the parts of the decomposition matrix $[S^\lambda:D^\mu]$ for the partitions with at most two parts are as follows.
\begin{center}
$\begin{matrix}
(17)\\
(15,2)\\
(13,4)\\
(11,6)\\
(9,8)
\end{matrix}\begin{pmatrix}
1&&&\\
1&1&&\\
0&1&1&\\
0&1&1&1\\
0&1&0&1&1
\end{pmatrix}$, $\begin{matrix}
(16,1)\\
(14,3)\\
(12,5)\\
(10,7)
\end{matrix}\begin{pmatrix}
1&&\\
0&1&\\
0&0&1\\
0&1&0&1
\end{pmatrix}$.
\end{center}
In order to identity $\mathcal{L}_5$, it suffices to check the characters of Young $\mathbb{F}G_{17}$-modules:
\begin{center}
$\begin{aligned}
\mathsf{ch}\ Y^{(17)}&=\chi^{(17)}, \\
\mathsf{ch}\ Y^{(15,2)}&=\chi^{(17)}+\chi^{(15,2)},\\
\mathsf{ch}\ Y^{(13,4)}&=\chi^{(17)}+\chi^{(15,2)}+\chi^{(13,4)}, \\
\mathsf{ch}\ Y^{(11,6)}&=\chi^{(13,4)}+\chi^{(11,6)},\\
\mathsf{ch}\ Y^{(9,8)}&=\chi^{(17)}+\chi^{(15,2)}+\chi^{(13,4)}+\chi^{(11,6)}+\chi^{(9,8)}.
\end{aligned}$
\end{center}

For the case $S(2,19)$, the blocks of $\mathbb{F}G_{19}$ and the parts of the decomposition matrix $[S^\lambda:D^\mu]$ for the partitions with at most two parts are
\begin{center}
$\begin{matrix}
(19)\\
(17,2)\\
(15,4)\\
(13,6)\\
(11,8)
\end{matrix}\begin{pmatrix}
1&&\\
0&1&\\
1&0&1\\
0&0&0&1\\
0&0&1&0&1
\end{pmatrix}$, $\begin{matrix}
(18,1)\\
(16,3)\\
(14,5)\\
(12,7)\\
(10,9)
\end{matrix}\begin{pmatrix}
1&&&\\
1&1&&\\
0&1&1&\\
0&1&1&1\\
0&1&0&1&1
\end{pmatrix}$.
\end{center}
Also, by Theorem \ref{henke charac}, we have
\begin{center}
$\begin{aligned}
\mathsf{ch}\ Y^{(19)}&=\chi^{(19)}, \ \mathsf{ch}\ Y^{(15,4)}=\chi^{(19)}+\chi^{(15,4)}, \\
\mathsf{ch}\ Y^{(11,8)}&=\chi^{(19)}+\chi^{(15,4)}+\chi^{(11,8)};\\
\mathsf{ch}\ Y^{(18,1)}&=\chi^{(18,1)}, \ \mathsf{ch}\ Y^{(16,3)}=\chi^{(18,1)}+\chi^{(16,3)},\\
\mathsf{ch}\ Y^{(14,5)}&=\chi^{(18,1)}+\chi^{(16,3)}+\chi^{(14,5)}, \ \mathsf{ch}\ Y^{(12,7)}=\chi^{(14,5)}+\chi^{(12,7)},\\
\mathsf{ch}\ Y^{(10,9)}&=\chi^{(18,1)}+\chi^{(16,3)}+\chi^{(14,5)}+\chi^{(12,7)}+\chi^{(10,9)};\\
\mathsf{ch}\ Y^{(17,2)}&=\chi^{(17,2)}; \ \mathsf{ch}\ Y^{(13,6)}=\chi^{(13,6)}.
\end{aligned}$
\end{center}
Similar to the above, the basic algebra of $S(2,19)$ is isomorphic to $\mathcal{D}_3\oplus \mathcal{L}_5\oplus \mathbb{F}\oplus \mathbb{F}$, where $\mathcal{D}_3$ is a $\tau$-tilting finite algebra defined in subsection \ref{subsec-3.2}.
\end{proof}

\begin{proposition}
Let $p=2$.
\begin{description}\setlength{\itemsep}{-3pt}
  \item[(1)] If $r$ is even, the Schur algebra $S(2,r)$ is $\tau$-tilting infinite for any $r\geqslant 10$.
  \item[(2)] If $r$ is odd, the Schur algebra $S(2,r)$ is $\tau$-tilting infinite for any $r\geqslant 21$.
\end{description}
\end{proposition}
\begin{proof}
We denote by $\overline{S(2,r)}$ the basic algebra of $S(2,r)$ and we use Lemma \ref{EH-method} to determine the quiver of $\overline{S(2,r)}$. When we display the quiver of $\overline{S(2,r)}$, we usually replace each vertex by the partition $\lambda$ associated with the Young module $Y^\lambda$. Then, the quiver of $\overline{S(2,10)}$ is
\begin{center}
$\xymatrix@C=1cm@R=0.7cm{ &&(6,4) \ar@<0.5ex>[d]^{\ }\ar@<0.5ex>[r]^{\ }&(10) \ar@<0.5ex>[l]^{\ }\ar@<0.5ex>[d]^{\ }\\ (8,2)\ar@<0.5ex>[r]^{\ } &(7,3)\ar@<0.5ex>[r]^{\ }\ar@<0.5ex>[l]^{\ } &(5^2) \ar@<0.5ex>[u]^{\ }\ar@<0.5ex>[l]^{\ }\ar@<0.5ex>[r]^{\ }&(9,1) \ar@<0.5ex>[u]^{\ }\ar@<0.5ex>[l]^{\ }}$,
\end{center}
and the quiver of $\overline{S(2,21)}$ is
\begin{center}
$\xymatrix@C=1cm@R=0.7cm{(17,4)\ar@<0.5ex>[d]^{\ }&(13,8) \ar@<0.5ex>[d]^{\ }\ar@<0.5ex>[r]^{\ }&(21) \ar@<0.5ex>[l]^{\ }\ar@<0.5ex>[d]^{\ }&(14,7)&&(18,3)\\ (15,6)\ar@<0.5ex>[r]^{\ }\ar@<0.5ex>[u]^{\ } &(11,10) \ar@<0.5ex>[u]^{\ }\ar@<0.5ex>[l]^{\ }\ar@<0.5ex>[r]^{\ }&(19,2) \ar@<0.5ex>[u]^{\ }\ar@<0.5ex>[l]^{\ }&(20,1)\ar@<0.5ex>[r]^{\ }&(12,9)\ar@<0.5ex>[l]^{\ }\ar@<0.5ex>[r]^{\ }&(16,5)\ar@<0.5ex>[l]^{\ }}$.
\end{center}
Now, it is enough to say that $S(2,10)$ and $S(2,21)$ are $\tau$-tilting infinite by Lemma \ref{tau-tilting infinite quiver}. Hence, the statement follows from Lemma \ref{n+r}.
\end{proof}

\begin{proposition}
Let $p=2$.
\begin{description}\setlength{\itemsep}{-3pt}
  \item[(1)] The wild Schur algebra $S(3,5)$ is $\tau$-tilting finite.
  \item[(2)] The Schur algebra $S(n,r)$ is $\tau$-tilting infinite for any $n\geqslant 3$ and $r\geqslant 6$.
\end{description}
\end{proposition}
\begin{proof}
We consider the Young modules $Y^\lambda$ for partitions $\lambda$ with at most three parts. Then, Specht modules $S^\mu$ in the Specht filtration of $Y^\lambda$ and composition factors $D^\mu$ which appear in $Y^\lambda$ are also corresponding to the partitions with at most three parts.

(1) We show that the basic algebra of $S(3,5)$ is $\tau$-tilting finite. The group algebra $\mathbb{F}G_{5}$ contains only two blocks, i.e., the principal block $B_1$ and the block $B_2$ labeled by $2$-core $(2,1)$. The parts of the decomposition matrix $[S^\lambda:D^\mu]$ for the partitions in $B_1$ and $B_2$ with at most three parts are as follows.
\begin{center}
$B_1: \begin{matrix}
(5)\\
(3,2)\\
(3,1^2)\\
(2^2,1)\\
\end{matrix}\begin{pmatrix}
1&\\
1&1\\
2&1\\
1&1
\end{pmatrix}$, $B_2: \begin{matrix}
(4,1)
\end{matrix}\begin{pmatrix}
1
\end{pmatrix}$.
\end{center}
Combining with \cite[Proposition 5.8]{Erdmann-finite}, the basic algebra of $S(3,5)$ is isomorphic to $\mathcal{U}_4\oplus \mathbb{F}$, where $\mathcal{U}_4:=\mathbb{F}Q/I$ is presented by
\begin{equation}
Q: \xymatrix@C=1cm{\circ\ar@<0.5ex>[r]^{\alpha_1}&\circ\ar@<0.5ex>[l]^{\beta_1}\ar@<0.5ex>[r]^{\alpha_2}&\circ\ar@<0.5ex>[l]^{\beta_2}\ar@<0.5ex>[r]^{\alpha_3}&
\circ\ar@<0.5ex>[l]^{\beta_3}}\ \text{and}\ I: \left \langle \begin{matrix}
\alpha_1\beta_1, \alpha_2\beta_2, \alpha_1\alpha_2\alpha_3,\beta_3\beta_2\beta_1,\alpha_3\beta_3-\beta_2\alpha_2
\end{matrix}\right \rangle.
\end{equation}
Since $\beta_2\alpha_2+\beta_3\alpha_3$ and $\beta_2\beta_1\alpha_1\alpha_2$ are non-trivial central elements of $\mathcal{U}_4$, the number of support $\tau$-tilting modules of $\mathcal{U}_4$ is the same as $\mathcal{U}_4/\left \langle \beta_2\alpha_2, \beta_3\alpha_3,\beta_2\beta_1\alpha_1\alpha_2\right \rangle$ by Proposition \ref{center}. Then, by direct calculation, we have
\renewcommand\arraystretch{1.2}
\begin{center}
\begin{tabular}{c|ccccc||ccccccc}
$s$&$0$&$1$&$2$&$3$&$4$&$\#\mathsf{s\tau\text{-}tilt}\ \mathcal{U}_4$  \\ \hline
$a_s(\mathcal{U}_4)$&1&4&12&36&83&136 \\
\end{tabular}.
\end{center}
We can also verify the number 136 by the String Applet \cite{G-string applet}.

(2) We shall show that $S(3,6)$, $S(3,7)$ and $S(3,8)$ are $\tau$-tilting infinite. Then, the statement follows from Lemma \ref{N>n} and Lemma \ref{n+r}. As we are already familiar with the strategy of determining the radical series of Young modules and the basic algebras of Schur algebras, we may leave this heavy work to a computer and some mathematicians indeed did. Here, we refer to Carlson and Matthews's program \cite{CM-computer}.

(2.1) Let $B$ be the principal block of $\mathbb{F}G_{6}$, the part of the decomposition matrix $[S^\lambda:D^\mu]$ for the partitions in $B$ with at most three parts is of the form
\begin{center}
$\begin{matrix}
(6)\\
(5,1)\\
(4,2)\\
(4,1^2)\\
(3^2)\\
(2^3)
\end{matrix}\begin{pmatrix}
1&&\\
1&1&\\
1&1&1\\
2&1&1\\
1&0&1\\
1&0&1
\end{pmatrix}$.
\end{center}
Then, the quiver of $S_B=\mathsf{End}_{\mathbb{F}G_{6}}(\underset{\lambda\in B\cap\Omega(3,6)}{ \bigoplus } Y^\lambda)$ is as follows.
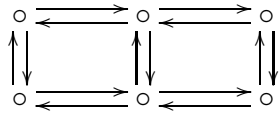
\begin{center}
$\xymatrix@C=1.2cm@R=0.7cm{ \circ\ar@<0.5ex>[d]^{\ }\ar@<0.5ex>[r]^{\ } &\circ \ar@<0.5ex>[l]^{\ } \ar@<0.5ex>[d]^{\ }\ar@<0.5ex>[r]^{\ }&\circ  \ar@<0.5ex>[l]^{\ }\ar@<0.5ex>[d]^{\ }\\ \circ \ar@<0.5ex>[r]^{\ } \ar@<0.5ex>[u]^{\ } &\circ  \ar@<0.5ex>[u]^{\ }\ar@<0.5ex>[l]^{\ }\ar@<0.5ex>[r]^{\ }&\circ  \ar@<0.5ex>[u]^{\ }\ar@<0.5ex>[l]^{\ }}$
\end{center}

(2.2) Let $B$ be the principal block of $\mathbb{F}G_{7}$, the part of the decomposition matrix $[S^\lambda:D^\mu]$ for the partitions in $B$ with at most three parts is of the form
\begin{center}
$\begin{matrix}
(7)\\
(5,2)\\
(4,2,1)\\
(5,1^2)\\
(3^2,1)\\
(3,2^2)
\end{matrix}\begin{pmatrix}
1&&\\
0&1&\\
1&1&1\\
1&1&0\\
1&0&1\\
1&0&1
\end{pmatrix}$.
\end{center}
Then, the quiver of $S_B=\mathsf{End}_{\mathbb{F}G_{7}}(\underset{\lambda\in B\cap\Omega(3,7)}{ \bigoplus } Y^\lambda)$ is as follows.
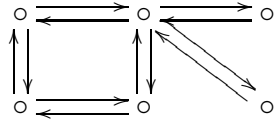
\begin{center}
$\xymatrix@C=1.2cm{ \circ\ar@<0.5ex>[d]^{\ }\ar@<0.5ex>[r]^{\ } &\circ \ar@<0.5ex>[l]^{\ }\ar@<0.5ex>[dr]^{\ } \ar@<0.5ex>[d]^{\ }\ar@<0.5ex>[r]^{\ }&\circ  \ar@<0.5ex>[l]^{\ } \\ \circ \ar@<0.5ex>[r]^{\ } \ar@<0.5ex>[u]^{\ } &\circ  \ar@<0.5ex>[u]^{\ }\ar@<0.5ex>[l]^{\ } &\circ  \ar@<0.5ex>[ul]^{\ }}$
\end{center}

(2.3) Let $B$ be the principal block of $\mathbb{F}G_{8}$, the part of the decomposition matrix $[S^\lambda:D^\mu]$ for the partitions in $B$ with at most three parts is of the form
\begin{center}
$\begin{matrix}
(8)\\
(7,1)\\
(6,2)\\
(5,3)\\
(4,3,1)\\
(6,1^2)\\
(4^2)\\
(4,2^2)\\
(3^2,2)
\end{matrix}\begin{pmatrix}
1&&&&\\
1&1&&&\\
0&1&1&&\\
0&1&1&1\\
2&1&1&1&1\\
1&1&1&0&0\\
0&1&0&1&0\\
2&0&1&0&1\\
2&0&0&0&1
\end{pmatrix}$.
\end{center}
Then, the quiver of $S_B=\mathsf{End}_{\mathbb{F}G_{8}}(\underset{\lambda\in B\cap\Omega(3,8)}{ \bigoplus } Y^\lambda)$ is as follows.
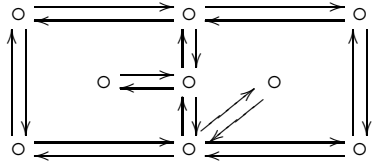
\begin{center}
$\xymatrix@C=0.7cm@R=0.5cm{\circ \ar@<0.5ex>[dd]^{\ }\ar@<0.5ex>[rr]^{\ } &&\circ\ar@<0.5ex>[d]^{\ }\ar@<0.5ex>[ll]^{\ }\ar@<0.5ex>[rr]^{\ }&
&\circ\ar@<0.5ex>[ll]^{\ }\ar@<0.5ex>[dd]^{\ } \\
 &\circ \ar@<0.5ex>[r]^{\ }&\circ\ar@<0.5ex>[d]^{\ }\ar@<0.5ex>[u]^{\ }\ar@<0.5ex>[l]^{\ }&\circ\ar@<0.5ex>[dl]^{\ }& \\
\circ\ar@<0.5ex>[rr]^{\ }\ar@<0.5ex>[uu]^{\ }&& \circ\ar@<0.5ex>[ur]^{\ }\ar@<0.5ex>[u]^{\ }\ar@<0.5ex>[ll]^{\ }\ar@<0.5ex>[rr]^{\ }& &\circ\ar@<0.5ex>[ll]^{\ }\ar@<0.5ex>[uu]^{\ }}$
\end{center}
By Lemma \ref{tau-tilting infinite quiver}, we conclude that $S(3,6)$, $S(3,7)$ and $S(3,8)$ are $\tau$-tilting infinite.
\end{proof}

\begin{corollary}
Let $p=2$. The wild Schur algebra $S(4,5)$ is $\tau$-tilting finite.
\end{corollary}
\begin{proof}
We consider the Young modules $Y^\lambda$ for partitions $\lambda$ of $5$ with at most four parts. Note that $S(3,5)$ is an idempotent truncation of $S(4,5)$ as we mentioned in Lemma \ref{N>n}. Compared with the case $S(3,5)$, the case $S(4,5)$ has only one additional partition $(2,1^3)$ which appears in the block of $\mathbb{F}G_{5}$ labeled by $2$-core (2,1). Then, the basic algebra of $S(4,5)$ is isomorphic to $\mathcal{U}_4\oplus \mathcal{A}_2$ based on the result of $S(3,5)$.
\end{proof}

\begin{proposition}
Let $p=2$. The Schur algebra $S(n,4)$ is $\tau$-tilting infinite for any $n\geqslant 4$.
\end{proposition}
\begin{proof}
By our strategy in subsection \ref{sub2.1}, one can see that $S(n,4)$ with $n\geqslant 5$ is always Morita equivalent to $S(4,4)$. So, it is enough to show that $S(4,4)$ is $\tau$-tilting infinite. In fact, the quiver of the basic algebra of $S(4,4)$ displayed below implies our statement.
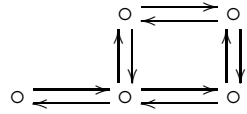
\begin{center}
$\xymatrix@C=1cm@R=0.7cm{ &\circ \ar@<0.5ex>[d]^{}\ar@<0.5ex>[r]^{}&\circ  \ar@<0.5ex>[l]^{}\ar@<0.5ex>[d]^{}\\ \circ \ar@<0.5ex>[r]^{} &\circ  \ar@<0.5ex>[u]^{}\ar@<0.5ex>[l]^{}\ar@<0.5ex>[r]^{ }&\circ  \ar@<0.5ex>[u]^{ }\ar@<0.5ex>[l]^{}}$
\end{center}
This quiver has been given by Xi in \cite{Xi-schur}.
\end{proof}

Now, we have already determined the $\tau$-tilting finiteness of wild Schur algebras over $p=2$, except for $S(2,r)$ with $r=8,17,19$, $S(3,4)$ and $S(n,5)$ with $n\geqslant 5$.

\begin{remark}\label{S(5,5) over p=2}
Let $p=2$. The Schur algebra $S(n,5)$ with $n\geqslant 6$ is always Morita equivalent to $S(5,5)$. Moreover, the basic algebra of $S(5,5)$ is isomorphic to $\mathcal{N}_5\oplus \mathcal{A}_2$ following \cite[Proposition 3.8]{Xi-schur}, where $\mathcal{N}_5:=\mathbb{F}Q/I$ is presented by
\begin{equation}
\begin{aligned}
\ & \ \ \ \ \ \ \ \  \ \ \ Q: \xymatrix@C=1cm@R=0.3cm{\circ\ar@<0.5ex>[r]^{\alpha_1}&\circ\ar@<0.5ex>[l]^{\beta_1}\ar@<0.5ex>[r]^{\alpha_2}&\circ\ar@<0.5ex>[l]^{\beta_2}\ar@<0.5ex>[r]^{\alpha_3}&
\circ\ar@<0.5ex>[l]^{\beta_3}\ar@<0.5ex>[r]^{\alpha_4}&\circ\ar@<0.5ex>[l]^{\beta_4}} \ \text{with}\\ I:& \left \langle \begin{matrix}
\alpha_1\beta_1, \alpha_2\beta_2, \alpha_3\beta_3, \beta_4\alpha_4,\alpha_1\alpha_2\alpha_3\alpha_4, \beta_4\beta_3\beta_2\beta_1, \beta_2\alpha_2-\alpha_3\alpha_4\beta_4\beta_3, \\
\alpha_2\alpha_3\alpha_4\beta_4-\beta_1\alpha_1\alpha_2\alpha_3, \beta_3\beta_2\beta_1\alpha_1-\alpha_4\beta_4\beta_3\beta_2
\end{matrix}\right \rangle.
\end{aligned}
\end{equation}
\end{remark}

\subsection{The characteristic $p=3$}
We assume in this subsection that the characteristic of $\mathbb{F}$ is 3. Then, the $\tau$-tilting finiteness for $S(n,r)$ is shown in Table \ref{p=3} and the proof is divided into the propositions displayed below. Here, we use the same conventions with Table \ref{p=2}.
\begin{table}
\caption{The $\tau$-tilting finite $S(n,r)$ over $p=3$.}\label{p=3}
\centering
\renewcommand\arraystretch{1.5}
\begin{tabular}{|c|c|c|c|c|c|c|c|c|c|c|c|c|c|c}
\hline
\diagbox{$n$}{$r$}&1&2&3&4&5&6&7&8&9&10&11&12&13&$\cdots$\\ \hline
2&\cellcolor{myblue}S&\cellcolor{myblue}S&\cellcolor{myblue}F&\cellcolor{myblue}F&\cellcolor{myblue}F&\cellcolor{myblue}F&\cellcolor{myblue}F&\cellcolor{myblue}F&
\cellcolor{myblue}T&\cellcolor{myblue}T&\cellcolor{myblue}T&\cellcolor{myred}W&\cellcolor{myred}W&\cellcolor{myred}$\cdots$\\ \hline
3&\cellcolor{myblue}S&\cellcolor{myblue}S&\cellcolor{myblue}F&\cellcolor{myblue}F&\cellcolor{myblue}F&\cellcolor{myred}W&\cellcolor{myblue}T&
\cellcolor{myblue}T&\cellcolor{myred}W&\cellcolor{myred}W&\cellcolor{myred}W&\cellcolor{myred}W&\cellcolor{myred}W&\cellcolor{myred}$\cdots$ \\ \hline
4&\cellcolor{myblue}S&\cellcolor{myblue}S&\cellcolor{myblue}F&\cellcolor{myblue}F&\cellcolor{myblue}F&\cellcolor{myred}W&\cellcolor{myred}W&\cellcolor{myred}W&\cellcolor{myred}W&\cellcolor{myred}W&\cellcolor{myred}W&\cellcolor{myred}W&
\cellcolor{myred}W&\cellcolor{myred}$\cdots$ \\ \hline
5&\cellcolor{myblue}S&\cellcolor{myblue}S&\cellcolor{myblue}F&\cellcolor{myblue}F&\cellcolor{myblue}F&\cellcolor{myred}W&\cellcolor{myred}W&\cellcolor{myred}W&\cellcolor{myred}W&\cellcolor{myred}W&\cellcolor{myred}W&\cellcolor{myred}W&
\cellcolor{myred}W&\cellcolor{myred}$\cdots$ \\ \hline
$\vdots$&\cellcolor{myblue}$\vdots$&\cellcolor{myblue}$\vdots$&\cellcolor{myblue}$\vdots$&\cellcolor{myblue}$\vdots$&\cellcolor{myblue}$\vdots$&\cellcolor{myred}$\vdots$&\cellcolor{myred}$\vdots$&\cellcolor{myred}$\vdots$&\cellcolor{myred}$\vdots$&\cellcolor{myred}$\vdots$&
\cellcolor{myred}$\vdots$&\cellcolor{myred}$\vdots$&\cellcolor{myred}$\vdots$&\cellcolor{myred}$\ddots$
\end{tabular}
\end{table}

\begin{proposition}
Let $p=3$. The Schur algebra $S(2,r)$ is $\tau$-tilting infinite for $r\geqslant 12$.
\end{proposition}
\begin{proof}
We show that both $S(2,12)$ and $S(2,13)$ are $\tau$-tilting infinite. Then, the statement follows from Lemma \ref{n+r}. In fact, let $B$ be the principal block of $\mathbb{F}G_{12}$ and the quiver of $S_B=\mathsf{End}_{\mathbb{F}G_{12}}(\underset{\lambda\in B\cap\Omega(2,12)}{ \bigoplus } Y^\lambda)$ following Lemma \ref{EH-method} is
\begin{center}
$\xymatrix@C=1cm@R=0.7cm{ &(11,1) \ar@<0.5ex>[d]^{\ }\ar@<0.5ex>[r]^{\ }&(12) \ar@<0.5ex>[l]^{\ }\ar@<0.5ex>[d]^{\ }\\ (9,3)\ar@<0.5ex>[r]^{\ } &(8,4) \ar@<0.5ex>[u]^{\ }\ar@<0.5ex>[l]^{\ }\ar@<0.5ex>[r]^{\ }&(6^2) \ar@<0.5ex>[u]^{\ }\ar@<0.5ex>[l]^{\ }}$,
\end{center}
where we replace each vertex by the partition $\lambda$ associated with $Y^\lambda$. Thus, $S(2,12)$ is $\tau$-tilting infinite by Lemma \ref{tau-tilting infinite quiver}. One can check that $S(2,13)$ also contains a $\tau$-tilting infinite subquiver as shown above.
\end{proof}

In the following, we refer to \cite{CM-computer} for the quiver of $S_B$ without further notice.

\begin{proposition}
Let $p=3$. Then, $S(3,6)$ and $S(3,r)$ with $r\geqslant 9$ are $\tau$-tilting infinite.
\end{proposition}

\begin{proof}
It suffices to show that $S(3,6)$, $S(3,10)$ and $S(3,11)$ are $\tau$-tilting infinite.

(1) Let $B$ be the principal block of $\mathbb{F}G_{6}$, the part of the decomposition matrix $[S^\lambda:D^\mu]$ for the partitions in $B$ with at most three parts is of the form
\begin{center}
$\begin{matrix}
(6)\\
(5,1)\\
(4,1^2)\\
(3^2)\\
(3,2,1)\\
(2^3)
\end{matrix}\begin{pmatrix}
1&&&&\\
1&1&&&\\
0&1&1&&\\
0&1&0&1&\\
1&1&1&1&1\\
1&0&0&0&1
\end{pmatrix}$.
\end{center}
Then, the quiver of $S_B=\mathsf{End}_{\mathbb{F}G_{6}}(\underset{\lambda\in B\cap\Omega(3,6)}{ \bigoplus } Y^\lambda)$ is as follows.
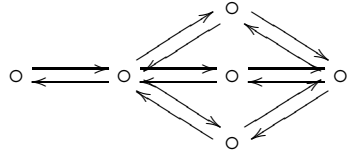
\begin{center}
$\xymatrix@C=1cm@R=0.5cm{ &&\circ\ar@<0.5ex>[dl]^{\ }\ar@<0.5ex>[dr]^{\ }&\\
\circ\ar@<0.5ex>[r]^{\ } &\circ \ar@<0.5ex>[dr]^{\ }\ar@<0.5ex>[ur]^{\ }\ar@<0.5ex>[l]^{\ }\ar@<0.5ex>[r]^{\ }&\circ\ar@<0.5ex>[r]^{\ }\ar@<0.5ex>[l]^{\ }&\circ\ar@<0.5ex>[l]^{\ }\ar@<0.5ex>[ul]^{\ }\ar@<0.5ex>[dl]^{\ }\\
&&\circ\ar@<0.5ex>[ur]^{\ }\ar@<0.5ex>[ul]^{\ }&}$
\end{center}

(2) Let $B_1$ be the principal block of $\mathbb{F}G_{10}$ and $B_2$ the block of $\mathbb{F}G_{11}$ labeled by $3$-core $(1^2)$, the parts of the decomposition matrix $[S^\lambda:D^\mu]$ for the partitions in $B_1$ and $B_2$ with at most three parts are of the form
\begin{center}
$B_1: \begin{matrix}
(10)\\
(8,2)\\
(7,3)\\
(7,2,1)\\
(5^2)\\
(4,3^2)
\end{matrix}\begin{pmatrix}
1&&&&&\\
1&1&&&&\\
0&1&1&&&\\
1&1&1&1&&\\
0&0&1&0&1&\\
0&0&1&1&1&1
\end{pmatrix}$, $B_2: \begin{matrix}
(10,1)\\
(9,2)\\
(7,4)\\
(7,2^2)\\
(6,5)\\
(4^2,3)
\end{matrix}\begin{pmatrix}
1&&&&&\\
1&1&&&&\\
0&1&1&&&\\
1&1&1&1&&\\
0&0&1&0&1&\\
0&0&1&1&1&1
\end{pmatrix}$.
\end{center}
Then, both the quiver of $S_{B_1}$ and $S_{B_2}$ are as follows.
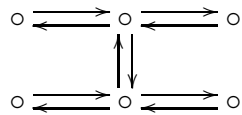
\begin{center}
$\xymatrix@C=1cm@R=0.7cm{\circ\ar@<0.5ex>[r]^{\ } &\circ\ar@<0.5ex>[l]^{\ }\ar@<0.5ex>[r]^{\ } \ar@<0.5ex>[d]^{\ }   &\circ\ar@<0.5ex>[l]^{\ } \\
\circ\ar@<0.5ex>[r]^{\ } &\circ\ar@<0.5ex>[r]^{\ } \ar@<0.5ex>[l]^{\ } \ar@<0.5ex>[u]^{\ } &\circ\ar@<0.5ex>[l]^{\ } }$
\end{center}

By Lemma \ref{tau-tilting infinite quiver}, the above two cases are $\tau$-tilting infinite quivers.
\end{proof}

\begin{proposition}
Let $p=3$. Then, $S(n,r)$ is $\tau$-tilting infinite for any $n\geqslant 4$ and $r\geqslant 6$.
\end{proposition}
\begin{proof}
Based on the result of $S(3,r)$, Lemma \ref{N>n} and Lemma \ref{n+r}, it suffices to show that $S(4,7)$ and $S(4,8)$ are $\tau$-tilting infinite.

(1) Let $B$ be the principal block of $\mathbb{F}G_{7}$, the part of the decomposition matrix $[S^\lambda:D^\mu]$ for the partitions in $B$ with at most four parts is of the form
\begin{center}
$\begin{matrix}
(7)\\
(5,2)\\
(4,3)\\
(4,2,1)\\
(3,2,1^2)\\
(4,1^3)\\
(2^3,1)
\end{matrix}\begin{pmatrix}
1&&&&\\
1&1&&&\\
0&1&1&&\\
1&1&1&1&\\
1&0&1&1&1\\
0&0&0&1&0\\
1&0&0&0&1
\end{pmatrix}$.
\end{center}
Then, the quiver of $S_B=\mathsf{End}_{\mathbb{F}G_{7}}(\underset{\lambda\in B\cap\Omega(4,7)}{ \bigoplus } Y^\lambda)$ is as follows.
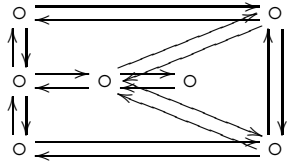
\begin{center}
$\xymatrix@C=0.7cm@R=0.5cm{\circ\ar@<0.5ex>[rrr]^{\ }\ar@<0.5ex>[d]^{\ } &&&\circ\ar@<0.5ex>[lll]^{\ }\ar@<0.5ex>[dll]^{\ }\ar@<0.5ex>[dd]^{\ } \\
\circ\ar@<0.5ex>[u]^{\ }\ar@<0.5ex>[d]^{\ }\ar@<0.5ex>[r]^{\ }&\circ\ar@<0.5ex>[urr]^{\ }\ar@<0.5ex>[drr]^{\ }\ar@<0.5ex>[r]^{\ } \ar@<0.5ex>[l]^{\ } &\circ\ar@<0.5ex>[l]^{\ }&\\
\circ\ar@<0.5ex>[rrr]^{\ }\ar@<0.5ex>[u]^{\ }&&&\circ\ar@<0.5ex>[lll]^{\ }\ar@<0.5ex>[ull]^{\ } \ar@<0.5ex>[uu]^{\ }}$
\end{center}

(2) Let $B$ be the block of $\mathbb{F}G_8$ labeled by $3$-core $(1^2)$, the part of the decomposition matrix $[S^\lambda:D^\mu]$ for the partitions in $B$ with at most four parts is of the form
\begin{center}
$\begin{matrix}
(7,1)\\
(6,2)\\
(4^2)\\
(4,2^2)\\
(3,2^2,1)
\end{matrix}\begin{pmatrix}
1&&&&\\
1&1&&&\\
0&1&1&&\\
1&1&1&1&\\
1&0&0&1&1
\end{pmatrix}$.
\end{center}
Then, the quiver of $S_B=\mathsf{End}_{\mathbb{F}G_{8}}(\underset{\lambda\in B\cap\Omega(4,8)}{ \bigoplus } Y^\lambda)$ is as follows.
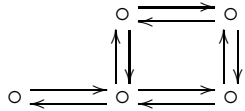
\begin{center}
$\xymatrix@C=1cm@R=0.7cm{ &\circ \ar@<0.5ex>[d]^{\ }\ar@<0.5ex>[r]^{\ }&\circ  \ar@<0.5ex>[l]^{\ }\ar@<0.5ex>[d]^{\ }\\ \circ \ar@<0.5ex>[r]^{\ } &\circ  \ar@<0.5ex>[u]^{\ }\ar@<0.5ex>[l]^{\ }\ar@<0.5ex>[r]^{\ }&\circ  \ar@<0.5ex>[u]^{\ }\ar@<0.5ex>[l]^{\ }}$
\end{center}

Obviously, $S(4,7)$ and $S(4,8)$ are $\tau$-tilting infinite.
\end{proof}

Hence, we have determined the $\tau$-tilting finiteness for all the cases over $p=3$.

\subsection{The characteristic $p\geqslant 5$}
The situation on $p\geqslant 5$ is much easier than the situation on $p=2,3$. As shown in Proposition \ref{summary}, tame Schur algebras do not appear in this case. Then, the $\tau$-tilting finiteness for $S(n,r)$ is shown in Table \ref{p=5} and the proof is divided into two propositions. Also, we use the same conventions with Table \ref{p=2}.
\begin{table}
\caption{The $\tau$-tilting finite $S(n,r)$ over $p\geqslant 5$.}\label{p=5}
\centering
\renewcommand\arraystretch{1.5}
\begin{tabular}{|c|c|c|c|c|c|c|c|c|c|c|c|c|c|c}
\hline
\diagbox{$n$}{$r$}&$1\sim p-1$&$p\sim 2p-1$&$2p\sim p^2-1$&$p^2\sim p^2+p-1$&$p^2+p\sim \infty$\\ \hline
2&\cellcolor{myblue}S&\cellcolor{myblue}F&\cellcolor{myblue}F&\cellcolor{mygray}W&\cellcolor{myred}W\\ \hline
3&\cellcolor{myblue}S&\cellcolor{myblue}F&\cellcolor{myred}W&\cellcolor{myred}W&\cellcolor{myred}W\\ \hline
4&\cellcolor{myblue}S&\cellcolor{myblue}F&\cellcolor{myred}W&\cellcolor{myred}W&\cellcolor{myred}W \\ \hline
5&\cellcolor{myblue}S&\cellcolor{myblue}F&\cellcolor{myred}W&\cellcolor{myred}W&\cellcolor{myred}W\\ \hline
$\vdots$&\cellcolor{myblue}$\vdots$&\cellcolor{myblue}$\vdots$&\cellcolor{myred}$\vdots$&\cellcolor{myred}$\vdots$&\cellcolor{myred}$\vdots$
\end{tabular}
\end{table}

\begin{proposition}
Let $p\geqslant 5$. Then, $S(2,r)$ is $\tau$-tilting infinite for any $r\geqslant p^2+p$.
\end{proposition}
\begin{proof}
It suffices to consider $S(2,p^2+p)$ and $S(2,p^2+p+1)$ following Lemma \ref{n+r}. To show the $\tau$-tilting finiteness of $S(2,p^2+p)$, we choose four partitions
\begin{center}
$(p^2+p), (p^2+p-1,1), (p^2-p,2p), (p^2-1, p+1)$,
\end{center}
which are contained in the principal block $B$ of $\mathbb{F}G_{p^2+p}$. By Lemma \ref{EH-method}, one may construct the following subquiver of the quiver of $S_B$.
\begin{center}
$\xymatrix@C=1.2cm{ v^{(p^2+p)} \ar@<0.5ex>[d]^{}\ar@<0.5ex>[r]^{}&v^{(p^2+p-1,1)} \ar@<0.5ex>[l]^{}\ar@<0.5ex>[d]^{}\\ v^{(p^2-p,2p)}  \ar@<0.5ex>[u]^{}\ar@<0.5ex>[r]^{ }&v^{(p^2-1, p+1)} \ar@<0.5ex>[u]^{ }\ar@<0.5ex>[l]^{}}$
\end{center}
This is just the $\tau$-tilting infinite quiver $\mathsf{Q}_1$ and therefore, $S(2,p^2+p)$ is $\tau$-tilting infinite.

Moreover, we can show that $S(2,p^2+p+1)$ contains the $\tau$-tilting infinite quiver $\mathsf{Q}_1$ as a subquiver if we choose $(p^2+p+1), (p^2+p-1,2), (p^2-p+1,2p)$ and $(p^2-1, p+2)$.
\end{proof}

\begin{proposition}
Let $p\geqslant 5$. Then, the Schur algebra $S(n,r)$ is $\tau$-tilting infinite for any $n\geqslant 3$ and $r\geqslant 2p$.
\end{proposition}
\begin{proof}
It suffices to consider $S(3,r)$ for $r=2p+x$ with $0\leqslant x\leqslant 2$. Let $B$ be the principal block of $\mathbb{F}G_{r}$. Then, the part of the decomposition matrix $[S^\lambda:D^\mu]$ for the partitions in $B$ with at most three parts is of the form
\begin{center}
$\begin{matrix}
(2p+x)\\
(2p-1,1+x)\\
(p+x,p)\\
(2p-2,1+x,1)\\
(p+x,p-1,1)\\
((p-1)^2,2+x)
\end{matrix}\begin{pmatrix}
1&&&&\\
1&1&&&\\
0&1&1&&\\
0&1&0&1&\\
1&1&1&1&1\\
1&0&0&0&1&1
\end{pmatrix}$.
\end{center}
We recall from \cite[Proposition 5.3.1]{Erdmann-finite} that the quiver of $S_B$ is
\begin{center}
$\vcenter{\xymatrix@C=1cm@R=0.5cm{ &&\circ\ar@<0.5ex>[dl]^{\ }\ar@<0.5ex>[dr]^{\ }&\\
\circ\ar@<0.5ex>[r]^{\ } &\circ \ar@<0.5ex>[dr]^{\ }\ar@<0.5ex>[ur]^{\ }\ar@<0.5ex>[l]^{\ }\ar@<0.5ex>[r]^{\ }&\circ\ar@<0.5ex>[r]^{\ }\ar@<0.5ex>[l]^{\ }&\circ\ar@<0.5ex>[l]^{\ }\ar@<0.5ex>[ul]^{\ }\ar@<0.5ex>[dl]^{\ }\\
&&\circ\ar@<0.5ex>[ur]^{\ }\ar@<0.5ex>[ul]^{\ }&}}$.
\end{center}
Then, the statement follows from Lemma \ref{N>n}, Lemma \ref{n+r} and Lemma \ref{tau-tilting infinite quiver}.
\end{proof}

Hence, we have already determined the $\tau$-tilting finiteness of wild Schur algebras over $p\geqslant 5$, except for $S(2,r)$ with $p^2\leqslant r\leqslant p^2+p-1$.

\subsection{The remaining cases}\label{remaining cases}
So far, we have determined the $\tau$-tilting finiteness of Schur algebras, except for the cases in $(\star)$. As we explained in Proposition \ref{S(2,8)over p=2} and Remark \ref{S(5,5) over p=2}, when $p=2$, the cases $S(2,r)$ for $r=17,18$ actually reduce to the principal block of $S(2,8)$ and the cases $S(n,5)$ for $n\geqslant 6$ can be reduced to the principal block of $S(5,5)$. Actually, the cases $S(2,r)$ for $p^2\leqslant r\leqslant p^2+p-1$ over $p\geqslant 5$ also reduce to a block case if one applies \cite[Theorem 13]{EH-two-blocks}, which says that any two blocks of $S(2,r)$ with the same number of simples are Morita equivalent over the same field. We will deal with them separately in the future.

Now, we may look at $S(3,4)$ over $p=2$. We recall from \cite[3.6]{DEMN-tame schur} that the basic algebra of $S(3,4)$ is presented by the bound quiver algebra $\mathcal{M}_4:=\mathbb{F}Q/I$ with
\begin{equation}
Q: \xymatrix@C=1.2cm@R=0.8cm{\circ\ar@<0.5ex>[r]^{\alpha_1 } &\circ\ar@<0.5ex>[r]^{\alpha_2 } \ar@<0.5ex>[l]^{\beta_1 } \ar@<0.5ex>[d]^{\alpha_3 } &\circ\ar@<0.5ex>[l]^{\beta_2 } \\
&\circ\ar@<0.5ex>[u]^{\beta_3 }   &} \ \text{and}\ I: \left \langle \begin{matrix}
\alpha_1\beta_1, \beta_3\alpha_3, \alpha_1\alpha_2,\beta_2\beta_1,\\
\alpha_1\alpha_3\beta_3, \alpha_3\beta_3\beta_1, \beta_1\alpha_1-\alpha_2\beta_2
\end{matrix}\right \rangle.
\end{equation}
We have the following partial results. (Recall that $a_s(A)$ is the number of pairwise non-isomorphic basic support $\tau$-tilting $A$-modules $M$ with $|M|=s$ for $0\leqslant s\leqslant \left | A \right |$.)
\renewcommand\arraystretch{1.2}
\begin{center}
\begin{tabular}{c|cccccccccccc}
$s$&$0$&$1$&$2$&$3$ \\ \hline
$a_s(\mathcal{M}_4)$&1&4&12&40 \\
\end{tabular}.
\end{center}

On the other hand, let $\mathcal{P}_4$ be the preprojective algebra of type $D_4$ which has been studied by Mizuno \cite{Mizuno}. We recall his result as follows.
\renewcommand\arraystretch{1.2}
\begin{center}
\begin{tabular}{c|ccccc||ccccccc}
$s$&$0$&$1$&$2$&$3$&$4$&$\#\mathsf{s\tau\text{-}tilt}\ \mathcal{P}_4$  \\ \hline
$a_s(\mathcal{P}_4)$&1&4&12&40&135&192\\
\end{tabular}.
\end{center}
Then, we expect that $a_4(\mathcal{M}_4)\leqslant 135$ according to our experiences in the calculation process. More generally, we expect that the following conjecture is true.

Let $\left | A\right |=n$ and $P_i$ the indecomposable projective $A$-modules. The Cartan matrix $[c_{ij}^A]$ for $i,j \in \{1,2,\dots, n\}$ is defined by $c_{ij}^A=\mathsf{dim}\ \mathsf{Hom}_A(P_i,P_j)$. Suppose that $A:=\mathbb{F}Q/I_1$ and $B:=\mathbb{F}Q/I_2$ are two algebras given by the same quiver $Q$ with different admissible ideals $I_1$ and $I_2$. It is obvious that $a_0(A)=a_0(B)=1$ and $a_1(A)=a_1(B)=\left | Q_0\right |$.

\begin{conjecture}\label{conjecture-cartan}
Assume that $A$ and $B$ have the same quiver and $B$ is $\tau$-tilting finite. If $c_{ij}^A\leqslant c_{ij}^B$ for any $i,j \in \{1,2,\dots, \left | Q_0 \right |\}$, then $a_s(A)\leqslant a_s(B)$ for $2\leqslant s\leqslant \left | Q_0 \right |$.
\end{conjecture}

For example, let $Q$ be the following quiver.
\begin{center}
$\xymatrix@C=1cm@R=0.7cm{1\ar@<0.5ex>[r]^{ }&2\ar@<0.5ex>[l]^{ }\ar@<0.5ex>[d]^{ }\ar@<0.5ex>[r]^{ }&3\ar@<0.5ex>[l]^{ }\\
&4\ar@<0.5ex>[u]^{ }&}$
\end{center}
We consider the preprojective algebra $\mathcal{P}_4$ and the tame block algebra $\mathcal{D}_4$ (or $\mathcal{H}_4$) defined in subsection \ref{subsec-3.2}. Then, we have
\begin{center}
$[c_{ij}^{\mathcal{D}_4}]=\begin{pmatrix}
1&1&0&1\\
1&3&1&2\\
0&1&2&0\\
1&2&0&2
\end{pmatrix}$ and $[c_{ij}^{\mathcal{P}_4}]=\begin{pmatrix}
2&2&1&1\\
2&4&2&2\\
1&2&2&1\\
1&2&1&2
\end{pmatrix}$.
\end{center}
It is easy to check that $c_{ij}^{\mathcal{D}_4}\leqslant c_{ij}^{\mathcal{P}_4}$ for any $i,j \in \{1,2,3,4\}$. Moreover, we have shown that $a_2(\mathcal{D}_4)=12=a_2(\mathcal{P}_4)$, $a_3(\mathcal{D}_4)=36<40=a_3(\mathcal{P}_4)$ and $a_4(\mathcal{D}_4)=61<135=a_4(\mathcal{P}_4)$.

Besides, Conjecture \ref{conjecture-cartan} is proved to be true by \cite[Theorem 1.1]{W-two-point} for the two-point algebras presented by the following quivers with some relations.
\begin{center}
$\xymatrix@C=1cm{\circ \ar[r]^{ }&\circ\ar@(ur,dr)^{ }}$, $\xymatrix@C=1cm{\circ &\circ\ar[l]^{ }\ar@(ur,dr)^{ }}$.
\end{center}

Finally, we have the following conjecture on Schur algebras.
\begin{conjecture}\label{conjecture star}
All wild Schur algebras contained in $(\star)$ are $\tau$-tilting finite.
\end{conjecture}

\appendix
\section{A complete list of $\tau$-tilting $\widetilde{\mathcal{D}_4}$-modules in the proof of Lemma \ref{method of D_4}.}
We recall the indecomposable projective $\widetilde{\mathcal{D}_4}$-modules $P_i$ as follows.
\begin{center}
$P_1=\begin{smallmatrix}
1\\
3\\
2
\end{smallmatrix},
P_2=\begin{smallmatrix}
2\\
3\\
1\\
3
\end{smallmatrix},
P_3=\begin{smallmatrix}
&3&\\
1&2&4\\
3\\
2
\end{smallmatrix}$, $P_4=\begin{smallmatrix}
4\\
3
\end{smallmatrix}$.
\end{center}
Then, we construct three indecomposable $\widetilde{\mathcal{D}_4}$-modules to describe some basic $\tau$-tilting modules. We first consider the $\tau$-tilting $\widetilde{\mathcal{D}_4}$-module $P_1\oplus P_2\oplus P_3\oplus P_4$ and take an exact sequence with a minimal left $\mathsf{add}(P_1\oplus P_2\oplus P_4)$-approximation $\pi_1$ of $P_3$:
\begin{center}
$P_3\overset{\pi_1}{\longrightarrow}P_1\oplus P_2\oplus P_4\longrightarrow \mathsf{coker}\ \pi_1\longrightarrow 0$.
\end{center}
We define $M_1:=\mathsf{coker}\ \pi_1$ and $P_1\oplus P_2\oplus M_1\oplus P_4$ is again a $\tau$-tilting $\widetilde{\mathcal{D}_4}$-module. Then, we take an exact sequence with a minimal left $\mathsf{add}(P_2\oplus M_1\oplus P_4)$-approximation $\pi_2$:
\begin{center}
$P_1\overset{\pi_2}{\longrightarrow}P_2\oplus M_1\longrightarrow \mathsf{coker}\ \pi_2\longrightarrow 0$
\end{center}
and define $M_2:=\mathsf{coker}\ \pi_2$. Last, we consider the $\tau$-tilting $\widetilde{\mathcal{D}_4}$-module $\substack{1\\3\\2}\oplus \substack{\ \\1\\3\\2}\substack{3\\\ \\\ \\ \  }\substack{\ \\4\\\ \\ \ }\oplus \substack{3\\1\\3\\2}\oplus \substack{\ \\1\\3\\2}\substack{3\\\ \\\ \\ \  }\substack{\ \\2\\\ \\ \ }$ (one may check this by the definition) and define $M_3:=\mathsf{coker}\ \pi_3$ as the cokernel of $\pi_3$, where $\pi_3$ is a minimal left $\mathsf{add}(\substack{\ \\1\\3\\2}\substack{3\\\ \\\ \\ \  }\substack{\ \\4\\\ \\ \ }\oplus \substack{3\\1\\3\\2}\oplus \substack{\ \\1\\3\\2}\substack{3\\\ \\\ \\ \  }\substack{\ \\2\\\ \\ \ })$-approximation with the following exact sequence.
\begin{center}
$\substack{1\\3\\2}\overset{\pi_3}{\longrightarrow}\substack{\ \\1\\3\\2}\substack{3\\\ \\\ \\ \  }\substack{\ \\4\\\ \\ \ }\oplus \substack{\ \\1\\3\\2}\substack{3\\\ \\\ \\ \  }\substack{\ \\2\\\ \\ \ }\longrightarrow \mathsf{coker}\ \pi_3\longrightarrow 0$.
\end{center}
Using $P_1, P_2, P_3, P_4, M_1, M_2, M_3$ and other explicitly described modules, we can give a complete list of $\tau$-tilting $\widetilde{\mathcal{D}_4}$-modules by direct computation of left mutations.
\begin{center}
\renewcommand\arraystretch{2.0}
\begin{tabular}{c|c|c|c}
\hline\hline
$P_1\oplus P_2\oplus P_3\oplus P_4$ & $\substack{\\2\\}\substack{3\\\ \\\ }\substack{\\4\\}\oplus P_2\oplus P_3\oplus P_4$ & $P_1\oplus \substack{\ \\1\\3\\2}\substack{3\\\ \\\ \\ \  }\substack{\ \\4\\\ \\ \ }\oplus P_3\oplus P_4$  & $P_1\oplus P_2\oplus M_1\oplus P_4$ \\ \hline

$P_1\oplus P_2\oplus P_3\oplus \substack{\ \\1\\3\\2}\substack{3\\\ \\\ \\ \  }\substack{\ \\2\\\ \\ \ }$& $\substack{\\2\\}\substack{3\\\ \\\ }\substack{\\4\\}\oplus \substack{\ \\1\\3\\2}\substack{3\\\ \\\ \\ \  }\substack{\ \\4\\\ \\ \ }\oplus P_3\oplus P_4$&$\substack{\\2\\}\substack{3\\\ \\\ }\substack{\\4\\}\oplus P_2\oplus \substack{2\\3}\oplus P_4$&$\substack{\\2\\}\substack{3\\\ \\\ }\substack{\\4\\}\oplus P_2\oplus P_3\oplus \substack{\ \\1\\3\\2}\substack{3\\\ \\\ \\ \  }\substack{\ \\2\\\ \\ \ }$ \\ \hline

$P_1\oplus \substack{\ \\1\\3\\2}\substack{3\\\ \\\ \\ \  }\substack{\ \\4\\\ \\ \ }\oplus \substack{1\\3}\oplus P_4$ &$P_1\oplus P_2\oplus M_1\oplus \substack{1\\ \\}\substack{ \\\ \\3}\substack{2\\ \\}$&$P_1\oplus \substack{1\\ \\}\substack{ \\\ \\3}\substack{4\\ \\}\oplus M_1\oplus P_4$& $M_2\oplus P_2\oplus M_1\oplus P_4$\\ \hline

$P_1\oplus \substack{\ \\1\\3\\2}\substack{3\\\ \\\ \\ \  }\substack{\ \\4\\\ \\ \ }\oplus P_3\oplus \substack{\ \\1\\3\\2}\substack{3\\\ \\\ \\ \  }\substack{\ \\2\\\ \\ \ }$ & $\substack{\\2\\}\substack{3\\\ \\\ }\substack{\\4\\}\oplus \substack{\ \\1\\3\\2}\substack{3\\\ \\\ \\ \  }\substack{\ \\4\\\ \\ \ }\oplus \substack{3\\4}\oplus P_4$& $M_2\oplus P_2\oplus \substack{2\\3}\oplus P_4$&$\substack{\\2\\}\substack{3\\\ \\\ }\substack{\\4\\}\oplus P_2\oplus \substack{2\\3}\oplus \substack{3\\2}$ \\ \hline

$\substack{\\2\\}\substack{3\\\ \\\ }\substack{\\4\\}\oplus P_2\oplus \substack{3\\2}\oplus \substack{\ \\1\\3\\2}\substack{3\\\ \\\ \\ \  }\substack{\ \\2\\\ \\ \ }$&$\substack{\ \\1\\3}\substack{3\\\ \\\ }\substack{\ \\4\\\ }\oplus \substack{\ \\1\\3\\2}\substack{3\\\ \\\ \\ \  }\substack{\ \\4\\\ \\ \ }\oplus \substack{1\\3}\oplus P_4$&$P_1\oplus \substack{1\\ \\}\substack{ \\\ \\3}\substack{4\\ \\}\oplus \substack{1\\3}\oplus P_4$&$P_1\oplus \substack{\ \\1\\3\\2}\substack{3\\\ \\\ \\ \  }\substack{\ \\4\\\ \\ \ }\oplus \substack{1\\3}\oplus \substack{3\\1\\3\\2}$\\ \hline

$M_2\oplus P_2\oplus M_1\oplus \substack{1\\ \\}\substack{ \\\ \\3}\substack{2\\ \\}$&$P_1\oplus \substack{1\\ \\}\substack{ \\\ \\3}\substack{4\\ \\}\oplus M_1\oplus \substack{1\\ \\}\substack{ \\\ \\3}\substack{2\\ \\}$ & $\substack{2\\ \\}\substack{ \\\ \\3}\substack{4\\ \\}\oplus \substack{1\\ \\}\substack{ \\\ \\3}\substack{4\\ \\}\oplus M_1\oplus P_4$ & $M_2\oplus \substack{2\\ \\}\substack{ \\\ \\3}\substack{4\\ \\}\oplus M_1\oplus P_4$\\ \hline

$P_1\oplus \substack{\ \\1\\3\\2}\substack{3\\\ \\\ \\ \  }\substack{\ \\4\\\ \\ \ }\oplus \substack{3\\1\\3\\2}\oplus \substack{\ \\1\\3\\2}\substack{3\\\ \\\ \\ \  }\substack{\ \\2\\\ \\ \ }$&$\substack{\\2\\}\substack{3\\\ \\\ }\substack{\\4\\}\oplus \substack{\ \\1\\3\\2}\substack{3\\\ \\\ \\ \  }\substack{\ \\4\\\ \\ \ }\oplus P_3\oplus \substack{\ \\1\\3\\2}\substack{3\\\ \\\ \\ \  }\substack{\ \\2\\\ \\ \ }$&$\substack{\ \\1\\3}\substack{3\\\ \\\ }\substack{\ \\4\\\ }\oplus \substack{\ \\1\\3\\2}\substack{3\\\ \\\ \\ \  }\substack{\ \\4\\\ \\ \ }\oplus \substack{3\\4}\oplus P_4$ & $M_2\oplus \substack{2\\ \\}\substack{ \\\ \\3}\substack{4\\ \\}\oplus \substack{2\\3}\oplus P_4$\\ \hline

$M_2\oplus P_2\oplus \substack{2\\3}\oplus \substack{2\\3\\1\\\ }\substack{\ \\\ \\\ \\3}\substack{\ \\\ \\2\\ \ }$ & $\substack{\ \\1\\3}\substack{3\\\ \\\ }\substack{\ \\4\\\ }\oplus \substack{\ \\1\\3\\2}\substack{3\\\ \\\ \\ \  }\substack{\ \\4\\\ \\ \ }\oplus \substack{1\\3}\oplus \substack{3\\1\\3\\2}$& $P_1\oplus \substack{1\\ \\}\substack{ \\\ \\3}\substack{4\\ \\}\oplus \substack{1\\3}\oplus S_1$& $M_2\oplus \substack{2\\ \\}\substack{ \\\ \\3}\substack{4\\ \\}\oplus M_1\oplus \substack{1\\ \\}\substack{ \\\ \\3}\substack{2\\ \\}$ \\ \hline

$M_2\oplus P_2\oplus \substack{2\\3\\1\\\ }\substack{\ \\\ \\\ \\3}\substack{\ \\\ \\2\\ \ }\oplus \substack{1\\ \\}\substack{ \\\ \\3}\substack{2\\ \\}$ & $\substack{2\\ \\}\substack{ \\\ \\3}\substack{4\\ \\}\oplus \substack{1\\ \\}\substack{ \\\ \\3}\substack{4\\ \\}\oplus M_1\oplus \substack{1\\ \\}\substack{ \\\ \\3}\substack{2\\ \\}$ & $P_1\oplus \substack{1\\ \\}\substack{ \\\ \\3}\substack{4\\ \\}\oplus S_1\oplus \substack{1\\ \\}\substack{ \\\ \\3}\substack{2\\ \\}$ & $\substack{2\\ \\}\substack{ \\\ \\3}\substack{4\\ \\}\oplus \substack{1\\ \\}\substack{ \\\ \\3}\substack{4\\ \\}\oplus S_4 \oplus P_4$\\ \hline

$M_3\oplus \substack{\ \\1\\3\\2}\substack{3\\\ \\\ \\ \  }\substack{\ \\4\\\ \\ \ }\oplus \substack{3\\1\\3\\2}\oplus \substack{\ \\1\\3\\2}\substack{3\\\ \\\ \\ \  }\substack{\ \\2\\\ \\ \ }$ & $\substack{\\2\\}\substack{3\\\ \\\ }\substack{\\4\\}\oplus \substack{\ \\1\\3\\2}\substack{3\\\ \\\ \\ \  }\substack{\ \\4\\\ \\ \ }\oplus M_3\oplus \substack{\ \\1\\3\\2}\substack{3\\\ \\\ \\ \  }\substack{\ \\2\\\ \\ \ }$ & $\substack{\ \\1\\3}\substack{3\\\ \\\ }\substack{\ \\4\\\ }\oplus \substack{\ \\1\\3\\2}\substack{3\\\ \\\ \\ \  }\substack{\ \\4\\\ \\ \ }\oplus \substack{3\\4}\oplus \substack{3\\1\\3\\2}$ & $M_2\oplus \substack{2\\ \\}\substack{ \\\ \\3}\substack{4\\ \\}\oplus \substack{2\\3}\oplus \substack{2\\3\\1\\\ }\substack{\ \\\ \\\ \\3}\substack{\ \\\ \\2\\ \ }$\\ \hline

$\substack{\ \\1\\3}\substack{3\\\ \\\ }\substack{\ \\4\\\ }\oplus \substack{3\\1\\3}\oplus \substack{1\\3}\oplus \substack{3\\1\\3\\2}$ & $M_2\oplus \substack{2\\ \\}\substack{ \\\ \\3}\substack{4\\ \\}\oplus \substack{2\\3\\1\\\ }\substack{\ \\\ \\\ \\3}\substack{\ \\\ \\2\\ \ }\oplus \substack{1\\ \\}\substack{ \\\ \\3}\substack{2\\ \\}$ & $\substack{2\\ \\}\substack{ \\\ \\3}\substack{4\\ \\}\oplus \substack{1\\ \\}\substack{ \\\ \\3}\substack{4\\ \\}\oplus \substack{1\\ \\}\substack{2 \\3}\substack{4\\ \\}\oplus \substack{1\\ \\}\substack{ \\\ \\3}\substack{2\\ \\}$ & $\substack{1\\ \\}\substack{2 \\3}\substack{4\\ \\}\oplus \substack{1\\ \\}\substack{ \\\ \\3}\substack{4\\ \\}\oplus S_1\oplus \substack{1\\ \\}\substack{ \\\ \\3}\substack{2\\ \\}$\\ \hline

$\substack{2\\ \\}\substack{ \\\ \\3}\substack{4\\ \\}\oplus \substack{1\\ \\}\substack{ \\\ \\3}\substack{4\\ \\}\oplus S_4 \oplus \substack{1\\ \\}\substack{2 \\3}\substack{4\\ \\}$ & $M_3\oplus \substack{3\\2}\oplus \substack{3\\1\\3\\2}\oplus \substack{\ \\1\\3\\2}\substack{3\\\ \\\ \\ \  }\substack{\ \\2\\\ \\ \ }$ & $M_3\oplus \substack{\ \\1\\3\\2}\substack{3\\\ \\\ \\ \  }\substack{\ \\4\\\ \\ \ }\oplus \substack{3\\1\\3\\2}\oplus \substack{3\\4}$ & $\substack{\\2\\}\substack{3\\\ \\\ }\substack{\\4\\}\oplus \substack{\ \\1\\3\\2}\substack{3\\\ \\\ \\ \  }\substack{\ \\4\\\ \\ \ }\oplus M_3\oplus \substack{3\\4}$\\ \hline

$\substack{\\2\\}\substack{3\\\ \\\ }\substack{\\4\\}\oplus \substack{3\\2}\oplus M_3\oplus \substack{\ \\1\\3\\2}\substack{3\\\ \\\ \\ \  }\substack{\ \\2\\\ \\ \ }$ & $\substack{\ \\1\\3}\substack{3\\\ \\\ }\substack{\ \\4\\\ }\oplus \substack{3\\1\\3}\oplus \substack{3\\4}\oplus \substack{3\\1\\3\\2}$ & $S_2\oplus \substack{2\\ \\}\substack{ \\\ \\3}\substack{4\\ \\}\oplus \substack{2\\3}\oplus \substack{2\\3\\1\\\ }\substack{\ \\\ \\\ \\3}\substack{\ \\\ \\2\\ \ }$ & $S_2\oplus \substack{2\\ \\}\substack{ \\\ \\3}\substack{4\\ \\}\oplus \substack{2\\3\\1\\\ }\substack{\ \\\ \\\ \\3}\substack{\ \\\ \\2\\ \ }\oplus \substack{1\\ \\}\substack{ \\\ \\3}\substack{2\\ \\}$ \\ \hline

$\substack{2\\ \\}\substack{ \\\ \\3}\substack{4\\ \\}\oplus S_2\oplus \substack{1\\ \\}\substack{2 \\3}\substack{4\\ \\}\oplus \substack{1\\ \\}\substack{ \\\ \\3}\substack{2\\ \\}$ & $\substack{1\\ \\}\substack{2 \\3}\substack{4\\ \\}\oplus S_2\oplus S_1\oplus \substack{1\\ \\}\substack{ \\\ \\3}\substack{2\\ \\}$ & $\substack{1\\ \\}\substack{2 \\3}\substack{4\\ \\}\oplus \substack{1\\ \\}\substack{ \\\ \\3}\substack{4\\ \\}\oplus S_1\oplus S_4$ & $\substack{2\\ \\}\substack{ \\\ \\3}\substack{4\\ \\}\oplus S_2\oplus S_4 \oplus \substack{1\\ \\}\substack{2 \\3}\substack{4\\ \\}$\\ \hline

$M_3\oplus \substack{3\\2}\oplus \substack{3\\1\\3\\2}\oplus \substack{3\\4}$ & $\substack{\\2\\}\substack{3\\\ \\\ }\substack{\\4\\}\oplus \substack{3\\2}\oplus M_3\oplus \substack{3\\4}$ & $S_3\oplus \substack{3\\1\\3}\oplus \substack{3\\4}\oplus \substack{3\\1\\3\\2}$ & $\substack{1\\ \\}\substack{2 \\3}\substack{4\\ \\}\oplus S_2\oplus S_1\oplus S_4$ \\ \hline

$S_3\oplus \substack{3\\2}\oplus \substack{3\\1\\3\\2}\oplus \substack{3\\4}$\\\hline\hline
\end{tabular}
\end{center}

\ \\

Department of Pure and Applied Mathematics, Graduate School of Information Science and Technology, Osaka University, 1-5 Yamadaoka, Suita, Osaka, 565-0871, Japan.

\emph{Email address}: \texttt{q.wang@ist.osaka-u.ac.jp}
\end{spacing}
\end{document}